\documentclass[12pt]{amsart}
\usepackage{amsthm}
\usepackage{amssymb}
\usepackage{geometry}
\usepackage{enumerate}
\usepackage{setspace}
\usepackage{pgfkeys}
\usepackage{upgreek}
\usepackage{tikz}
\usepackage{tikz-cd}
\usetikzlibrary{matrix, arrows}
\usepackage[unicode=true]
 {hyperref}
\hypersetup{
colorlinks=true,
urlcolor=black,
citecolor=blue,
linkcolor=blue,
}

\newtheorem{thm}{Theorem}[section]
\newtheorem{prop}[thm]{Proposition}
\newtheorem{lem}[thm]{Lemma}
\newtheorem{cor}[thm]{Corollary}

\theoremstyle{definition}
\newtheorem{definition}[thm]{Definition}

\theoremstyle{remark}
\newtheorem{remark}[thm]{Remark}

\numberwithin{equation}{section}

\begin{document}

\large 

\title[Galois Module Structure of Abelian Extensions]{On the Galois Module Structure of the Square Root of the Inverse Different in Abelian Extensions}

\author{Cindy (Sin Yi) Tsang}
\address{Department of Mathematics, University of California, Santa Barbara}
\email{cindytsy@math.ucsb.edu}
\urladdr{http://math.ucsb.edu/$\sim$cindytsy} 

\date{June 1, 2015}

\begin{abstract}Let $K$ be a number field with ring of integers $\mathcal{O}_K$ and $G$ a finite group of odd order. If $K_h$ is a weakly ramified $G$-Galois $K$-algebra, then its square root $A_h$ of the inverse different is a locally free $\mathcal{O}_{K}G$-module and hence determines a class in the locally free class group $\mbox{Cl}(\mathcal{O}_KG)$ of $\mathcal{O}_KG$. We show that for $G$ abelian and under suitable assumptions, the set of all such classes is a subgroup of $\mbox{Cl}(\mathcal{O}_KG)$.
\end{abstract}

\maketitle

\tableofcontents


\section{Introduction}\label{s:1}

Let $K$ be a number field with ring of integers $\mathcal{O}_K$ and $G$ a finite group. If $K_h$ is a tame $G$-Galois $K$-algebra with ring of integers $\mathcal{O}_h$, then a classical theorem of Noether implies that $\mathcal{O}_h$ is a locally free $\mathcal{O}_KG$-module and hence determines a class $\mbox{cl}(\mathcal{O}_h)$ in the locally free class group $\mbox{Cl}(\mathcal{O}_KG)$ of $\mathcal{O}_KG$. A class $c\in\mbox{Cl}(\mathcal{O}_KG)$ is called \emph{realizable} if $c=\mbox{cl}(\mathcal{O}_h)$ for some tame $G$-Galois $K$-algebra $K_h$. The set of all realizable classes is denoted by $R(\mathcal{O}_KG)$.

The set of isomorphism classes of $G$-Galois $K$-algebras may be parametrized by the pointed set $H^1(\Omega_K,G)$, where $\Omega_K$ denotes the absolute Galois group of $K$ whose action on $G$ is trivial. Let $H^1_t(\Omega_K,G)$ be the subset consisting of the tame $G$-Galois $K$-algebras and consider the map
\[
\mbox{gal}:H_t^1(\Omega_K,G)\longrightarrow\mbox{Cl}(\mathcal{O}_KG);\hspace{1em}\mbox{gal}(h)=\mbox{cl}(\mathcal{O}_h).
\]
Assume for the moment that $G$ is abelian, in which case $H^1_t(\Omega_K,G)$ is a group. Brinkhuis proved in \cite[Proposition 3.10]{B} that $\mbox{gal}$ is \emph{weakly multiplicative}, i.e. for all $h_1,h_2\in H_t^1(\Omega_K,G)$, we have
\[
\mbox{gal}(h_1h_2)=\mbox{gal}_1(h_1)\mbox{gal}_1(h_2)\hspace{1cm}\mbox{whenever }d(h_1)\cap d(h_2)=\emptyset,
\]
where $d(h)$ denotes the set of primes in $K$ which are ramified in $K_h/K$ given any $h\in H^1(\Omega_K,G)$. In addition, McCulloh proved in \cite[Corollary 6.20]{M} that the set $R(\mathcal{O}_KG)$ is in fact a subgroup of $\mbox{Cl}(\mathcal{O}_KG)$; this is not obvious because $\mbox{gal}$ is not a homomorphism in general.

In this paper, we study the Galois module structure of the square root of the inverse different instead of that of the ring of integers. First, recall that:

\begin{prop}\label{prop:1.1}Let $F$ be a finite extension of $\mathbb{Q}_p$ and $L/F$ a finite Galois extension. Let $\mathfrak{D}_{L/F}$ denote the different ideal of $L/F$ and $\mathfrak{P}$ the prime ideal in $\mathcal{O}_L$. Then, the highest power $v_L(\mathfrak{D}_{L/F})$ of $\mathfrak{P}$ dividing $\mathfrak{D}_{L/F}$ is given by
\[
v_L(\mathfrak{D}_{L/F})=\sum_{n=0}^{\infty}(|G_n|-1),
\]
where $G_n$ is the $n$-th ramification group of $L/F$.
\end{prop}
\begin{proof}
See \cite[Chapter IV Proposition 4]{S}, for example.
\end{proof}

Assume henceforth that $G$ has odd order. If $K_h$ is a $G$-Galois $K$-algebra, then Proposition~\ref{prop:1.1} implies that there exists an ideal $A_h$ in $K_h$ whose square is the inverse different of $K_h$. If $K_h$ is in addition weakly ramified, i.e. all of its second ramification groups are trivial, then it follows from \cite[Theorem 1 in Section 2]{E} that $A_h$ is a locally free $\mathcal{O}_KG$-module and thus defines a class $\mbox{cl}(A_h)$ in $\mbox{Cl}(\mathcal{O}_KG)$. A class $c\in\mbox{Cl}(\mathcal{O}_KG)$ is called \emph{$A$-realizable} if $c=\mbox{cl}(A_h)$ for some weakly ramified $G$-Galois $K$-algebra $K_h$, and \emph{tame $A$-realizable} if the above $K_h$ may be chosen to be tame. The sets of all $A$-realizable classes and tame $A$-realizable classes are denoted by $\mathcal{A}(\mathcal{O}_KG)$ and $\mathcal{A}^t(\mathcal{O}_KG)$, respectively. We make the remark that $\mathcal{A}^{t}(\mathbb{Z}G)=1$ as a consequence of \cite[Theorem 3]{E}, and that $\mathcal{A}(\mathbb{Z}G)=1$ when $G$ is abelian from \cite[Theorem 1.2]{V}.

As in the case of the ring of integers, we consider the map
\[
\mbox{gal}_A:H^1_w(\Omega_K,G)\longrightarrow\mbox{Cl}(\mathcal{O}_KG);\hspace{1em}\mbox{gal}_A(h)=\mbox{cl}(A_h),
\]
where $H^1_w(\Omega_K,G)$ is the subset of $H^1(\Omega_K,G)$ consisting of the weakly ramified $G$-Galois $K$-algebras. In Section~\ref{s:5}, assuming that $G$ is abelian, we show that the map $\mbox{gal}_A$ preserves inverses and that it is weakly multiplicative as is $\mbox{gal}$. More precisely, we prove that:

\begin{thm}\label{thm:1.2}Let $K$ be a number field and $G$ an abelian group of odd order. 
\begin{enumerate}[(a)]
\item If $h\in H^1_w(\Omega_K,G)$, then $h^{-1}\in H^1_w(\Omega_K,G)$ and
\[
\mbox{gal}_A(h^{-1})=\mbox{gal}_A(h)^{-1}.
\]
\item If $h_1,h_2\in H^1_w(\Omega_K,G)$ and $d(h_1)\cap d(h_2)=\emptyset$, then $h_1h_2\in H^1_w(\Omega_K,G)$ and 
\[
\mbox{gal}_A(h_1h_2)=\mbox{gal}_A(h_1)\mbox{gal}_A(h_2).
\]
\end{enumerate}
\end{thm}

Moreover, following the techniques developed by McCulloh in \cite{M}, analogous to \cite[Theorem 6.17 and Corollary 6.20]{M}, in Section~\ref{s:10} we prove that:

\begin{thm}\label{thm:1.3}Let $K$ be a number field and $G$ an abelian group of odd order.
\begin{enumerate}[(a)]
\item The set $\mathcal{A}^{t}(\mathcal{O}_KG)$ is a subgroup of $\mbox{Cl}(\mathcal{O}_KG)$.
\item Given $c\in \mathcal{A}^t(\mathcal{O}_KG)$ and a finite set $T$ of primes in $\mathcal{O}_K$, there is a tame $G$-Galois $K$-algebra $K_h$ such that $K_h$ is a field, every $v\in T$ is unramified in $K_h/K$, and $c=\mbox{cl}(A_h)$.
\end{enumerate}
\end{thm}

In Section~\ref{s:14}, we show $A_h$ does not give rise to a new $A$-realizable class in addition to the tame ones for wildly and weakly ramified $G$-Galois $K$-algebras $K_h$ satisfying suitable hypotheses. In particular, we prove that:

\begin{thm}\label{thm:1.4}Let $K$ be a number field and $G$ an abelian group of odd order. Let $K_h$ be a wildly and weakly ramified $G$-Galois $K$-algebra and $V$ the set of primes in $\mathcal{O}_K$ which are wildly ramified in $K_h/K$. If conditions (i) and (ii) below are satisfied by every $v\in V$, then $\mbox{cl}(A_h)\in\mathcal{A}^t(\mathcal{O}_KG)$. 
\begin{enumerate}[(i)]
\item The ramification index of $v$ over $\mathbb{Q}$ is one;
\item The ramification index of $v$ in $K_h/K$ is prime.
\end{enumerate}
\end{thm}

We would like to remove hypothesis (ii) in Theorem~\ref{thm:1.3}. Currently, we can only do so if we extend scalars to the maximal $\mathcal{O}_K$-order $\mathcal{M}(KG)$ in $KG$. To be precise, let $\mbox{Cl}(\mathcal{M}(KG))$ be the locally free class group of $\mathcal{M}(KG)$ and
\[
\Psi:\mbox{Cl}(\mathcal{O}_KG)\longrightarrow\mbox{Cl}(\mathcal{M}(KG))
\]
the canonical homomorphism afforded by extension of scalars. In Section~\ref{s:17}, we prove that:

\begin{thm}\label{thm:1.5}Let $K$ be a number field and $G$ an abelian group of odd order. Let $K_h$ be a wildly and weakly ramified $G$-Galois $K$-algebra and $V$ the set of primes in $\mathcal{O}_K$ which are wildly ramified in $K_h/K$. If the ramification index of every $v\in V$ over $\mathbb{Q}$ is one, then $\Psi(\mbox{cl}(A_h))\in\Psi(\mathcal{A}^t(\mathcal{O}_KG))$.
\end{thm}

We remark that Theorems~\ref{thm:1.3} to~\ref{thm:1.5} above are the first known results in the literature concerning the structure of the sets $\mathcal{A}^t(\mathcal{O}_KG)$ and $\mathcal{A}(\mathcal{O}_KG)$.

Here is a brief outline of the contents of this paper. In Sections~\ref{s:2} and~\ref{s:3}, we give a brief review of Galois algebras, resolvends, and locally class groups. We then prove Theorem~\ref{thm:1.2} in Sections~\ref{s:4} and~\ref{s:5}. In Sections~\ref{s:6} and~\ref{s:7}, following \cite[Sections 2 and 4]{M}, we define reduced resolvends and the modified Stickelberger transpose map, which will be key in everything that follows. The crucial step in proving Theorems~\ref{thm:1.3} to~\ref{thm:1.5} is to characterize reduced resolvends of local generators of $A_h$ over $\mathcal{O}_KG$ for any weakly ramified $G$-Galois $K$-algebra $K_h$. We consider the case when $K_h$ is tame in Sections~\ref{s:8} to~\ref{s:10}, and the case when $K_h$ is wild in Sections~\ref{s:11} to~\ref{s:17}.

\noindent\textbf{Notation and Conventions.} Throughout the remaining of this paper, we fix a number field $K$ and an abelian group $G$ of odd order.

The symbol $F$ will denote an arbitrary field extension of either $\mathbb{Q}$ or $\mathbb{Q}_p$ for some prime number $p$. Given any such $F$, we adopt the following notation:
\begin{align*}
\mathcal{O}_F&:=\mbox{the ring of integers in $F$};\\
F^c&:=\mbox{a fixed algebraic closure of $F$};\\
\Omega_F&:=\mbox{Gal}(F^c/F);\\
F^t&:=\mbox{the maximal tamely ramified extension of $F$ in $F^c$};\\
\Omega_F^t&:=\mbox{Gal}(F^t/F);\\
M_F&:=\mbox{the set of all finite primes in $F$};\\
[-1]&:=\mbox{the involution on $F^cG$ induced by the}\\
&\hspace{0.81cm}\mbox{involution $s\mapsto s^{-1}$ on $G$};\\
\hat{G}_F&:=\mbox{the group of irreducible $F^c$-valued characters on $G$};\\
\zeta_{n,F}&:=\mbox{a chosen primitive $n$-th root of unity in $F^c$ for $n\in\mathbb{Z}^+$}
\end{align*}
Moreover, we will let $\Omega_F$ and $\Omega_F^t$ act trivially on $G$. 

For $F$ a number field and $v\in M_F$, we adopt the following notation:
\begin{align*}
F_v&:=\mbox{the completion of $F$ with respect to $v$};\\
i_v&:=\mbox{a fixed embedding $F^c\longrightarrow F_v^c$ extending the natural}\\
&\hspace{0.81cm}\mbox{embedding $F\longrightarrow F_v$};\\
\tilde{i}_v&:=\mbox{the embedding $\Omega_{F_v}\longrightarrow\Omega_{F}$ induced by $i_v$}.
\end{align*}
Moreover, for each $n\in\mathbb{Z}^+$, we take $i_v(\zeta_{n,F})$ to be the chosen primitive $n$-th root of unity in $F_v^c$, where $\zeta_{n,F}$ that in $F^c$.

For $F$ a finite extension of $\mathbb{Q}_p$, we use the following notation:
\begin{align*}
\pi_F&:=\mbox{a chosen uniformizer in $F$};\\
q_F&:=\mbox{the order of the residue field $\mathcal{O}_F/(\pi_F)$};\\
v_F&:=\mbox{the additive valuation $F\longrightarrow\mathbb{Z}$ such that $v_F(\pi_F)=1$}.
\end{align*}
Moreover, given a fractional $\mathcal{O}_F$-ideal $I$ in $F$, we will write 
\[
v_F(I):=\mbox{the highest power of $(\pi_F)$ dividing $I$}.
\]

Finally, we note that only finite primes will be considered in this paper, and that all of the cohomology considered is continuous.


\section{Galois Algebras and Resolvends}\label{s:2}

Let $F$ be a number field or finite extension of $\mathbb{Q}_p$. Following \cite[Section 1]{M}, we give a brief review of Galois algebras and resolvends of their elements.

\begin{definition}\label{def:2.1}A \emph{Galois $F$-algebra with group $G$} or \emph{$G$-Galois $F$-algebra} is a commutative semi-simple $F$-algebra $L$ on which $G$ acts on the left as a group of automorphisms with $L^{G}=F$ and $[L:F]=|G|$. Two $G$-Galois $F$-algebras are said to be \emph{isomorphic} if there is an $F$-algebra isomorphism between them which preserves the action of $G$.
\end{definition}

The set of isomorphism classes of $G$-Galois $F$-algebras is in one-one correspondence with the pointed set
\[
H^1(\Omega_F,G):=\mbox{Hom}(\Omega_F,G)/\mbox{Inn}(G).
\]
In particular, each $h\in\mbox{Hom}(\Omega_F,G)$ is associated to the $F$-algebra
\[
F_{h}:=\mbox{Map}_{\Omega_F}(^{h}G,F^{c}),
\]
where $^{h}G$ is the group $G$ endowed with the $\Omega_F$-action given by
\[
(\omega\cdot s):=h(\omega)s\hspace{1cm}\mbox{for $s\in G$ and $\omega\in\Omega_F$}.
\]
The $G$-action on $F_{h}$ is defined by
\[
(s\cdot a)(t):=a(ts)\hspace{1cm}\mbox{for $a\in F_h$ and $s,t\in G$}.
\]
It is evident that if $\{s_i\}$ is a set of right coset representatives of $h(\Omega_F)$ in $G$, then each $a\in F$ is determined by the values $a(s_i)$, and that each $a(s_i)$ may be arbitrarily chosen provided that it is fixed by all $\omega\in\ker(h)$. Hence, if
\[
F^{h}:=(F^{c})^{\ker(h)},
\]
then this choice $\{s_i\}$ of coset representatives induces an isomorphism
\[
F_{h}\simeq \prod_{h(\Omega_F)\backslash G}F^{h}
\]
of $F$-algebras. Since $h$ induces an isomorphism $\mbox{Gal}(F^h/F)\simeq h(\Omega_F)$, we have
\[
[F_h:F]=[G:h(\Omega_F)][F^h:F]=|G|.
\]
Viewing $F$ as embedded in $F_h$ as the constant $F$-valued functions, we easily see that $F_h^G=F$. Hence, indeed $F_h$ is a $G$-Galois $F$-algebra.

It is not hard to verify that every $G$-Galois $F$-algebra is isomorphic to $F_h$ for some $h\in\mbox{Hom}(\Omega_F,G)$, and that for $h,h'\in\mbox{Hom}(\Omega_F,G)$ we have $F_{h}\simeq F_{h'}$ if and only if $h$ and $h'$ differ by an element in $\mbox{Inn}(G)$. Since $G$ is abelian for us, this implies that the isomorphism classes of $G$-Galois $F$-algebras may be identified with the set $\mbox{Hom}(\Omega_F,G)$ and in particular has a group structure.

\begin{definition}\label{def:2.2}Given $h\in\mbox{Hom}(\Omega_F,G)$, we define
\begin{align*}
F^h&:=(F^c)^{\ker(h)};\\
\mathcal{O}^h&:=\mathcal{O}_{F^h};\\
A^h&:=A_{F^h/F},
\end{align*}
where $A_{F^h/F}$ is the square root of the inverse different of $F^h/F$, which exists because $G$ has odd order. Define the \emph{ring of integers of $F_h$} by
\[
\mathcal{O}_h:=\mbox{Map}_{\Omega_F}(^hG,\mathcal{O}^h)
\]
and the \emph{square root of the inverse different of $F_h/F$} by
\[
A_h:=\mbox{Map}_{\Omega_F}(^hG,A^h).
\]
\end{definition}

\begin{definition}\label{def:2.3}Given $h\in\mbox{Hom}(\Omega_F,G)$, we say that $F_h$ or $h$ is \emph{unramified} if $F^h/F$ is unramified. Similarly for \emph{tame}, \emph{weakly ramified}, and \emph{wildly ramified}.
\end{definition}

Observe that $h\in\mbox{Hom}(\Omega_F,G)$ is tame if and only if it factors through the quotient map $\Omega_F\longrightarrow\Omega_F^t$. In particular, the subgroup of $\mbox{Hom}(\Omega_F,G)$ consisting of all the tame homomorphisms may be identified with $\mbox{Hom}(\Omega_F^t,G)$.

\begin{remark}\label{rem:2.4}Let $F$ be a number field and $v\in M_F$. For $h\in\mbox{Hom}(\Omega_F,G)$, set
\[
h_v:=h\circ\tilde{i}_v\in\mbox{Hom}(\Omega_{F_v},G).
\]
It is proved in \cite[(1.4)]{M} that $(F_v)_{h_v}\simeq F_v\otimes_FF_h$, and
consequently we have
\[
A_{h_v}\simeq\mathcal{O}_{F_v}\otimes_{\mathcal{O}_F}A_h.
\]
\end{remark}

Next, consider the $F^c$-algebra $\mbox{Map}(G,F^c)$, on which we let $G$ act via
\[
(s\cdot a)(t):=a(ts)\hspace{1cm}\mbox{for $a\in\mbox{Map}(\Omega_F,G)$ and $s,t\in G$}.
\]
Note that $F_h$ is an $FG$-submodule of $\mbox{Map}(G,F^c)$ for all $h\in\mbox{Hom}(\Omega_F,G)$.

\begin{definition}\label{def:2.5}Define the \emph{resolvend map} by 
\[
\textbf{r}_{G}=\textbf{r}_{G,F}:\mbox{Map}(G,F^{c})\longrightarrow F^{c}G;\hspace{1em}
\textbf{r}_{G}(a):=\sum\limits _{s\in G}a(s)s^{-1}.
\]
\end{definition}

It is clear that $\textbf{r}_{G}$ is an isomorphism of $F^cG$-modules, but not an isomorphism of $F^cG$-algebras because it does not preserve multiplication. Furthermore, given $a\in\mbox{Map}(G,F^c)$ we have that $a\in F_h$ if and only if
\begin{equation}\label{eq:2.1}
\omega\cdot\textbf{r}_{G}(a)=\textbf{r}_{G}(a)h(\omega)
\hspace{1cm}\mbox{for all }\omega\in\Omega_F.
\end{equation}
In particular, if $\textbf{r}_{G}(a)$ is invertible then $h$ is given by
\[
h(\omega)=\textbf{r}_{G}(a)^{-1}(\omega\cdot\textbf{r}_{G}(a))
\hspace{1cm}\mbox{for all $\omega\in\Omega_F$}.
\]
Resolvends are useful for identifying elements $a\in F_h$ for which $F_h=FG\cdot a$ and elements $a\in A_h$ for which $A_h=\mathcal{O}_FG\cdot a$.

\begin{prop}\label{prop:2.6}
Let $a\in F_{h}$. Then $F_h=FG\cdot a$ if and only if 
\[
\textbf{r}_{G}(a)\in (F^{c}G)^{\times}.
\]
\end{prop}
\begin{proof}See \cite[Proposition 1.8]{M}.
\end{proof}

Recall that we have the standard algebra trace map
\[
Tr=Tr_F:\mbox{Map}(G,F^c)\longrightarrow F^c;\hspace{1em}Tr(a):=\sum_{s\in G}a(s).
\]
Via restriction, this yields a trace map $F_{h}\longrightarrow F$ for each $h\in\mbox{Hom}(\Omega_F,G)$, which we still denote by $Tr$ by abuse of notation.

\begin{definition}\label{def:2.7}Let $h\in\mbox{Hom}(\Omega_F,G)$ and $M$ an $\mathcal{O}_F$-lattice in $F_h$. The \emph{dual of $M$ with respect to $Tr$} is defined to be the $\mathcal{O}_F$-module
\[
M^*:=\{a\in F_h\mid Tr(aM)\subset\mathcal{O}_F\}
\]
We say that $M$ is \emph{self-dual} if $M=M^*$.
\end{definition}

It is well-known for any field extension $L/F$, its square root of the inverse different $A_{L/F}$ is self-dual with respect to the trace map of $L/F$ (this follows from \cite[Chapter 3 (2.14)]{FT}, for example). From this, we see that $A_h$ is self-dual for all $h\in\mbox{Hom}(\Omega_F,G)$. This fact will be important.

If $F_h=FG\cdot a$, then $\mathcal{O}_FG\cdot a$ is an $\mathcal{O}_F$-lattice in $F_h$ and so we may consider its dual. Resolvends may be used to detect whether $\mathcal{O}_FG\cdot a$ is self-dual. To that end, first recall that $[-1]$ denotes the involution on $F^cG$ induced by the involution $s\mapsto s^{-1}$ on $G$. A simple calculation shows that
\begin{equation}\label{eq:2.2}
\textbf{r}_{G}(a)\textbf{r}_{G}(b)^{[-1]}=\sum_{s\in G}Tr((s\cdot a)b)s^{-1}\hspace{1cm}\mbox{for }a,b\in F_h.
\end{equation}

\begin{prop}
\label{prop:2.8}Let $F_h=FG\cdot a$. Then $\mathcal{O}_FG\cdot a$ is self-dual if and only if
\[
\textbf{r}_G(a)\textbf{r}_G(a)^{[-1]}\in(\mathcal{O}_FG)^{\times}.
\]
\end{prop}
\begin{proof}Let $b\in F_{h}$ be such that $\{s\cdot b\mid s\in G\}$ is the dual basis of $\{s\cdot a \mid s\in G\}$ with respect to $Tr$, so we have $(\mathcal{O}_FG\cdot a)^{*}=\mathcal{O}_FG\cdot b$. This implies that $\mathcal{O}_FG\cdot a$ is self-dual if and only if $\mathcal{O}_FG\cdot a=\mathcal{O}_FG\cdot b$, which in turn is equivalent to
\[
\textbf{r}_G(a)\textbf{r}_G(b)^{-1}\in(\mathcal{O}_FG)^\times.
\]
But $\textbf{r}_G(b)^{-1}=\textbf{r}_G(a)^{[-1]}$ as a consequence of (\ref{eq:2.2}), so the claim follows. 
\end{proof}

\begin{cor}\label{cor:2.9}Let $a\in A_h$. Then $A_h=\mathcal{O}_FG\cdot a$ if and only if
\[
\textbf{r}_G(a)\textbf{r}_G(a)^{[-1]}\in(\mathcal{O}_FG)^{\times}.
\]
\end{cor}
\begin{proof}By Proposition~\ref{prop:2.6}, both statements imply that $\mathcal{O}_FG\cdot a$ is $F_h=FG\cdot a$. Assuming so, we may consider $(\mathcal{O}_FG\cdot a)^*$. Since $a\in A_h$, we have
\[
\mathcal{O}_FG\cdot a\subset A_{h}=A_{h}^{*}\subset(\mathcal{O}_FG\cdot a)^{*},
\]
which shows that $A_h=\mathcal{O}_FG\cdot a$ if and only if $\mathcal{O}_FG\cdot a$ is self-dual. The claim now follows from Proposition~\ref{prop:2.8}.
\end{proof}

\begin{cor}\label{cor:2.10}Let $a\in F_h$. Then, the two statements  $\mathcal{O}_h=\mathcal{O}_FG\cdot a$ and $h$ is unramified hold true simultaneously if and only if
\begin{equation}\label{eq:2.3}
\textbf{r}_G(a)\in(\mathcal{O}_{F^c}G)^{\times}.
\end{equation}
Moreover, for $F$ a finite extension of $\mathbb{Q}_p$, there always exists $a\in F_h$ satisfying (\ref{eq:2.3}) whenever $h$ is unramified.
\end{cor}
\begin{proof}Note that both $\mathcal{O}_h=\mathcal{O}_FG\cdot a$ and (\ref{eq:2.3}) implies that $a\in\mathcal{O}_h$. Assuming that this is the case, we see that $\textbf{r}_G(a)\in(\mathcal{O}_{F^c}G)^{\times}$ is equivalent to
\[
\textbf{r}_G(a)\textbf{r}_G(a)^{[-1]}\in(\mathcal{O}_{F}G)^{\times}
\]
because $[-1]$ induces an involution on $(\mathcal{O}_{F^c}G)^\times$. Since $a\in\mathcal{O}_h$ and $\mathcal{O}_h\subset A_h$, we deduce from Corollary~\ref{cor:2.9} that (\ref{eq:2.3}) occurs precisely when
\[
A_h=\mathcal{O}_FG\cdot a=\mathcal{O}_h,
\]
or equivalently, when $\mathcal{O}_h=\mathcal{O}_FG\cdot a$ and $h$ is unramified. The claim after (\ref{eq:2.3}) follows from a classical theorem of Noether or from \cite[Proposition 5.5]{M}.
\end{proof}


\section{Locally Free Class Groups}\label{s:3}

Let $F$ be a number field and $\mathfrak{A}$ an $\mathcal{O}_F$-order in $KG$. We recall certain basic facts concerning the locally free class group $\mbox{Cl}(\mathfrak{A})$ of $\mathfrak{A}$.

\begin{definition}\label{def:3.1}For each $v\in M_F$, define
\[
\mathfrak{A}_v:=\mathcal{O}_{F_v}\otimes_{\mathcal{O}_F}\mathfrak{A}.
\]
Similarly, given an $\mathfrak{A}$-lattice $M$ and $v\in M_F$, define
\[
M_v:=\mathcal{O}_{F_v}\otimes_{\mathcal{O}_F}M.
\]
Moreover, we write $\mbox{cl}(M)$ for the class determined by $M$ in $\mbox{Cl}(\mathfrak{A})$.
\end{definition}

\begin{definition}\label{def:3.2}
Let $J(FG)$ be the restricted direct product of $(F_vG)^\times$ with respect to the subgroups $\mathfrak{A}_v^\times$ as $v$ ranges over $M_F$. This is independent of the choice of $\mathcal{O}_F$-order $\mathfrak{A}$, for if $\mathfrak{A}'$ is another $\mathcal{O}_F$-order in $FG$, then $\mathfrak{A}_v=\mathfrak{A}'_v$ for almost all $v\in M_F$. Let
\[
\partial=\partial_F:(FG)^\times\longrightarrow J(FG)
\]
denote the diagonal map and
\[
U(\mathfrak{A}):=\prod_{v\in M_F}\mathfrak{A}_v
\]
the group of unit id\`{e}les.
\end{definition}

For each $c\in J(FG)$, we define
\[
\mathfrak{A}\cdot c:=\left(\bigcap_{v\in M_F}\mathfrak{A}_v\cdot c_v\right)\cap FG,
\]
which is a locally free $\mathfrak{A}$-module in $FG$.

\begin{thm}\label{thm:3.3}The map
\[
j_\mathfrak{A}:J(FG)\longrightarrow\mbox{Cl}(\mathfrak{A});\hspace{1em}j_\mathfrak{A}(c):=\mbox{cl}(\mathfrak{A}\cdot c),
\]
is a surjective homomorphism whose kernel is given by
\[
\ker(j_\mathfrak{A})=\partial(FG)^\times U(\mathfrak{A}).
\]
In particular, we have an isomorphism
\[
\mbox{Cl}(\mathfrak{A})\simeq\frac{J(FG)}{\partial(FG)^\times U(\mathfrak{A})}
\]
\end{thm}
\begin{proof}See \cite[Theorem 49.22 and Exercise 51.1]{C-R}, for example.
\end{proof}

\begin{prop}\label{prop:3.4}
Let $M$ be a locally free $\mathfrak{A}$-module. Suppose that
\[
F\otimes_{\mathcal{O}_F}M=FG\cdot b
\]
and that for each $v\in M_F$, we have
\[
M_v=\mathfrak{A}_v\cdot a_v.
\]
Since $b$ and $a_v$ are both free generators of $K_v\otimes_{\mathcal{O}_F}M$ over $F_vG$, we have
\[
a_v=c_v\cdot b\hspace{1cm}\mbox{for some $c_v\in (F_vG)^\times$ and for each $v\in M_F$}.
\]
Then, we have $c:=(c_v)\in J(FG)$ and $j_\mathfrak{A}(c)=\mbox{cl}(M)$.
\end{prop}
\begin{proof}Since $M$ and $\mathfrak{A}\cdot b$ are both $\mathfrak{A}$-lattices in $F\otimes_{\mathcal{O}_F}M$, we have $M_v=\mathfrak{A}_v\cdot b$ for all but finitely many $v\in M_F$, which implies that $c\in J(FG)$. Moreover, we have $j_\mathfrak{A}(c)=\mbox{cl}(M)$ because the $FG$-module isomoprhism
\[
K\otimes_{\mathcal{O}_F}M\longrightarrow FG\cdot b;\hspace{1em}\gamma\cdot b\mapsto\gamma
\]
restricts to an isomorphism $1\otimes_{\mathcal{O}_F}M\longrightarrow \mathfrak{A}\cdot c$ of $\mathfrak{A}$-modules.
\end{proof}

\begin{remark}\label{rem:3.5}For simplicity, we will write $j=j_F$ for $j_\mathfrak{A}$ when $\mathfrak{A}=\mathcal{O}_FG$. Given a weakly ramified $h\in\mbox{Hom}(\Omega_F,G)$, the Normal Basis Theorem implies that
\[
F_h=FG\cdot b\hspace{1cm}\mbox{for some }b\in F_h.
\]
Moreover, since $A_h$ is locally free over $\mathcal{O}_FG$, for each $v\in M_F$ we have
\[
A_{h_v}=\mathcal{O}_{F_v}G\cdot a_v\hspace{1cm}\mbox{for some }a_v\in A_{h_v}
\]
(c.f. Remark~\ref{rem:2.4}). Let $c\in J(FG)$ be as in Proposition~\ref{prop:3.4}, so $j(c)=\mbox{cl}(A_h)$. Observe that for each $v\in M_F$, since $\textbf{r}_G$ is an isomorphism of $F_v^cG$-modules, the equality $a_v=c_v\cdot b$ is equivalent to
\[
\textbf{r}_G(a_v)=c_v\cdot\textbf{r}_G(b).
\]
Such $\textbf{r}_G(b)$ is already described by Proposition~\ref{prop:2.6}. In this rest of this paper, we give explicit descriptions of $\textbf{r}_G(a_v)$ for $h_v$ satisfying suitable hypotheses.
\end{remark}


\section{Properties of Local Resolvends}\label{s:4}

Let $F$ be a finite extension of $\mathbb{Q}_p$. We prove two fundamental properties of resolvends $\textbf{r}_G(a)$ for any $a$ satisfying $A_h=\mathcal{O}_FG\cdot a$ for some weakly ramified $h\in\mbox{Hom}(\Omega_F,G)$. Recall that such $\textbf{r}_G(a)$ is invertible by Propostion~\ref{prop:2.6} and that $\textbf{r}_G$ is an isomorphism of $F^cG$-modules.

\begin{prop}\label{prop:4.1}Let $h\in\mbox{Hom}(\Omega_F,G)$ be weakly ramified. Then $h^{-1}$ is also weakly ramified. Moreover, assume that $A_h=\mathcal{O}_FG\cdot a$ and
\[
\textbf{r}_G(a')=\textbf{r}_G(a)^{-1}.
\]
Then $a'\in F_{h^{-1}}$ and  $A_{h^{-1}}=\mathcal{O}_FG\cdot a'$.
\end{prop}
\begin{proof}The fact that $a'\in F_{h^{-1}}$ follows from (\ref{eq:2.1}). Furthermore, since $G$ has odd order, we have $\ker(h)=\ker(h^{-1})$ so $A^h=A^{h^{-1}}$. It is then clear that $h^{-1}$ is weakly ramified if $h$ is. Moreover, Corollary~\ref{cor:2.9} implies that
\[
\textbf{r}_G(a')=\gamma\textbf{r}_G(a)^{[-1]}\hspace{1cm}\mbox{for some }\gamma\in(\mathcal{O}_FG)^\times
\]
Since $A^{h}=A^{h^{-1}}$, the above implies that $a'\in A_{h^{-1}}$. Since
\[
\textbf{r}_G(a')\textbf{r}_G(a')^{[-1]}
=\gamma\gamma^{[-1]}\textbf{r}_G(a)^{[-1]}\textbf{r}_G(a)\in(\mathcal{O}_FG)^\times,
\]
we deduce from Corollary~\ref{cor:2.9} that $A_{h^{-1}}=\mathcal{O}_FG\cdot a'$, as desired.
\end{proof}

\begin{lem}\label{lem:4.2}For a finite Galois extension $E/F$, let $G(E/F)_n$ denote the $n$-th ramification group of $E/F$. Let $U/F$ and $L/F$ be two finite Galois extensions with $U/F$ unramified.
\begin{enumerate}[(a)]
\item The homomorphism
\[
\mbox{Gal}(UL/F)\longrightarrow\mbox{Gal}(L/F);
\hspace{1em}\tau\mapsto\tau|_L
\]
induces an isomorphism $G(UL/F)_n\simeq G(L/F)_n$ for all $n\geq 0$.
\item If $L/F$ is weakly ramified and $M/F$ is a Galois subextension of $L/F$, then the extension $M/F$ is also weakly ramified.
\item If $L/F$ is abelian and $e_0:=|G(L/F)_0/G(L/F)_1|$, then
\[
G(L/F)_n=G(L/F)_{n+1}\hspace{1cm}\mbox{for all $n$ not divisible by $e_0$}.
\]
In particular, if $L/F$ is wildly and weakly ramified in addition, then
\[
G(L/F)_0=G(L/F)_1.
\]
\end{enumerate}
\end{lem}
\begin{proof}See \cite[Proposition 2.2]{V} for (a) and (b); the proof there is valid even if $F\neq\mathbb{Q}_p$. See \cite[Chapter IV Proposition 9 Corollary 2]{S} for (c).
\end{proof}

\begin{lem}\label{lem:4.3}Let $h,h_1,h_2\in\mbox{Hom}(\Omega_F,G)$ be such that
\begin{enumerate}[(i)]
\item $h=h_1h_2$;
\item $h_1$ is unramified.
\end{enumerate}
Then $F^h/F$ has the same ramification as $F^{h_2}/F$. Assume in addition that
\begin{enumerate}[(i)]\setcounter{enumi}{2}
\item $h_2$ is weakly ramified.
\end{enumerate}
Then $F^h/F$ is also weakly ramified and
\[
v_{F^h}(A^h)=v_{F^{h_2}}(A^{h_2}).
\]
\end{lem}
\begin{proof}Let $e$ be the ramification index of $F^{h_2}/F$. First, we show that we have the following diagram, where the numbers indicate the ramification indices.
\begin{equation}\label{eq:4.1}
\begin{tikzpicture}[baseline=(current bounding box.center)]
\node at (0,-2) [name=F] {$F$};
\node at (-2,0) [name=F1] {$F^{h_1}$};
\node at (0,0) [name=Fh] {$F^h$};
\node at (2,0) [name=F2] {$F^{h_2}$};
\node at (0,2) [name=L] {$F^{h_1}F^{h_2}$};
\path[font=\small]
(F) edge node[auto] {$1$} (F1)
(F) edge node[auto] {$e$} (Fh)
(F2) edge node[auto] {$e$} (F)
(F1) edge node[auto] {$e$} (L)
(Fh) edge node[auto] {$1$} (L)
(L) edge node[auto] {$1$} (F2);
\end{tikzpicture}
\end{equation}
Hypothesis (i) implies that $\ker(h_1)\cap\ker(h_2)=\ker(h_1)\cap\ker(h)$, and so
\[
F^{h_1}F^{h_2}
=(F^c)^{\ker(h_1)\cap\ker(h_2)}
=(F^c)^{\ker(h_1)\cap\ker(h)}
=F^{h_1}F^{h}.
\]
The extensions $F^{h_1h_2}/F^{h_2}$ and $F^{h_1}F^h/F^h$ are unramified because $F^{h_1}/F$ is unramified by hypothesis (ii). It is then clear that the ramification indices of $F^h/F$ and $F^{h_1h_2}/F^{h_1}$ are both be equal to $e$. This proves the first claim.

Next assume in addition that $F^{h_2}/F$ is weakly ramified. Since $F^{h_1}/F$ is unramified, it follows from Lemma~\ref{lem:4.2} (a) that $F^{h_1}F^{h_2}/F$ is also weakly ramified, and hence from Lemma~\ref{lem:4.2} (b) that $F^h/F$ is weakly ramified. Moreover, using the same notation as in Lemma~\ref{lem:4.2}, we know from Proposition~\ref{prop:1.1} that
\begin{align*}
v_{F^h}(A^h)&=(G(F^h/F)_0+G(F^h/F)_1-2)/2;\\
v_{F^{h_2}}(A^{h_2})&=(G(F^{h_2}/F)_0+G(F^{h_2}/F)_1-2)/2.
\end{align*}
If $(e,p)=1$, then $F^h/F$ and $F^{h_2}/F$ are both tame and
\[
G(F^h/F)_1=1=G(F^{h_2}/F)_1.
\]
If $(e,p)>1$, then $F^h/F$ and $F^{h_2}$ are both wildly and weakly ramified, and Lemma~\ref{lem:4.2} (c) yields 
\begin{align*}
G(F^h/F)_0&=G(F^h/F)_1;\\
G(F^{h_2}/F)_0&=G(F^{h_2}/F)_1.
\end{align*}
Since $G(F^h/F)_0$ and $G(F^{h_2}/F)_0$ both have order $e$, in either case we have that $v_{F^h}(A^h)=v_{F^{h_2}}(A^{h_2})$, as desired.
\end{proof}

\begin{prop}\label{prop:4.4}Let $h_1,h_2\in\mbox{Hom}(\Omega_F,G)$ with $h_1$ unramified and $h_2$ weakly ramified. Write $h:=h_1h_2$. Then $h$ is also weakly ramified. Moreover, assume that $A_{h_i}=FG\cdot a_i$ for $i=1,2$ and
\[
\mathbf{r}_{G}(a)=\mathbf{r}_{G}(a_1)\mathbf{r}_{G}(a_2).
\]
Then $a\in F_h$ and $A_h=\mathcal{O}_FG\cdot a$.
\end{prop}
\begin{proof}The fact that $a\in F_h$ follows from (\ref{eq:2.1}), and that $h$ is weakly ramified from Lemma~\ref{lem:4.3}. Moreover, Corollary~\ref{cor:2.9} implies that
\begin{equation}\label{eq:4.2}
\textbf{r}_G(a)\textbf{r}_G(a)^{[-1]}\in(\mathcal{O}_FG)^\times.
\end{equation}
A simple calculation then shows that
\[
a(s)=\sum_{rt=s}a_{1}(r)a_{2}(t)\hspace{1cm}\mbox{for all }s\in G.
\]
Note that $A_{h_1}=\mathcal{O}_{h_1}$ because $h_1$ is unramified. Since $a_1\in A_{h_1}$, we know that
\[
v_{F^{h_1}F^{h_2}}(a_{1}(r))\geq0\hspace{1cm}\mbox{for all }r\in G
\]
Similarly, because $a_2\in A^{h_2}$ and $F^{h_1}F^{h_2}/F^{h_2}$ is unramified from (\ref{eq:4.1}), we have
\[
v_{F^{h_1}F^{h_{2}}}(a_{2}(t))\geq v_{F^{h_2}}(A^{h_2})\hspace{1cm}\mbox{for all }t\in G.
\]
We then deduce that
\[
v_{F^{h_1}F^{h_2}}(a(s))\geq v_{F^{h_2}}(A^{h_2})\hspace{1cm}\mbox{for all }s\in G.
\]
But observe that $F^{h_1}F^{h_2}/F^h$ is unramified from (\ref{eq:4.1}) and $v_{F^{h_2}}(A^{h_2})=v_{F^h}(A^h)$ from Lemma~\ref{lem:4.3}, so the above implies that 
\[
v_{F^h}(a(s))\geq v_{F^h}(A^h)\hspace{1cm}\mbox{for all }s\in G.
\]
Since $A_h=\mbox{Map}_{\Omega_F}(^hG,A^h)$, this shows that $a\in A_h$. From Collorary~\ref{cor:2.9}, this together with (\ref{eq:4.2}) shows that $A_h=\mathcal{O}_FG\cdot a$, as desired.
\end{proof}

Proposition~\ref{prop:4.4} above will be used to prove that $\mbox{gal}_A$ is weakly multiplicative. It will also play an important role in the proofs of Theorems~\ref{thm:1.3} to~\ref{thm:1.5}.


\section{Proof of Theorem~\ref{thm:1.2}}\label{s:5}

\begin{proof}[Proof of Theorem~\ref{thm:1.2} (a)]
Let $h\in\mbox{Hom}(\Omega_K,G)$ be weakly ramified. The fact that $h^{-1}$ is also weakly ramified follows from Proposition~\ref{prop:4.1}. Let $K_h=KG\cdot b$ and $A_{h_v}=\mathcal{O}_{K_v}G\cdot a_v$ for $v\in M_K$ as in Remark~\ref{rem:3.5}. Moreover, let $c\in J(KG)$ be such that $a_v=c_v\cdot b$ for each $v\in M_K$, so $j(c)=\mbox{cl}(A_h)$.  Suppose that
\begin{align*}
\textbf{r}_G(b')&=\textbf{r}_G(b)^{-1};\\
\textbf{r}_G(a_v')&=\textbf{r}_G(a_v)^{-1}.
\end{align*}
Then $K_{h^{-1}}=KG\cdot b$ by (\ref{eq:2.1}) and Proposition~\ref{prop:2.6}, and $A_{h^{-1}_v}=\mathcal{O}_{K_v}G\cdot a_v'$ by Proposition~\ref{prop:4.1}. Note that for each $v\in M_K$, by definition we have
\[
\textbf{r}_G(a_v')=c_v^{-1}\textbf{r}_G(b'),
\]
which is equivalent to $a_v'=c_v^{-1}\cdot b$. It then follows from Proposition~\ref{prop:3.4} that  $j(c^{-1})=\mbox{cl}(A_{h^{-1}})$, as desired.
\end{proof}

\begin{proof}[Proof of Theorem~\ref{thm:1.2} (b)]
Let $h_1,h_2\in\mbox{Hom}(\Omega_K,G)$ be weakly ramified such that $d(h_1)\cap d(h_2)=\emptyset$. The fact that $h:=h_1h_2$ is also weakly ramified follows from Proposition~\ref{prop:4.1}. For $i=1,2$, let $K_{h_i}=KG\cdot b_i$ and $A_{h_{i,v}}=\mathcal{O}_{K_v}G\cdot a_{i,v}$ for $v\in M_K$ as in Remark~\ref{rem:3.5}. Moreover, let $c_i\in J(KG)$ be such that $a_{i,v}=c_{i,v}\cdot b$ for each $v\in M_K$, so $j(c_i)=\mbox{cl}(A_{h_i})$.  Suppose that
\begin{align*}
\textbf{r}_G(b)&=\textbf{r}_G(b_1)\textbf{r}_G(b_2);\\
\textbf{r}_G(a_v)&=\textbf{r}_G(a_{1,v})\textbf{r}_G(a_{2,v}).
\end{align*}
Then $K_{h}=KG\cdot b$ by (\ref{eq:2.1}) and Proposition~\ref{prop:2.6}, and $A_{h_v}=\mathcal{O}_{K_v}G\cdot a_v$ by Proposition~\ref{prop:4.4}. Note that for each $v\in M_K$, by definition we have
\[
\textbf{r}_G(a_v)=c_{1,v}c_{2,v}\textbf{r}_G(b),
\]
which is equivalently to $a_v=(c_{1,v}c_{2,v})\cdot b$. It then follows from Proposition~\ref{prop:3.4} that $j(c_1c_2)=\mbox{cl}(A_h)$, as desired.
\end{proof}


\section{Cohomology and Reduced Resolvends}\label{s:6}

Let $F$ be a number field or finite extension of $\mathbb{Q}_p$. Following \cite[Sections 1 and 2]{M}, we use cohomology to define reduced resolvends.

Recall that $\Omega_F$ acts trivially on $G$. Define
\[
\mathcal{H}(FG):=((F^{c}G)^{\times}/G)^{\Omega_F}
\]
and consider the short exact sequence
\begin{equation}\label{eq:6.1}
\begin{tikzcd}[column sep=1cm, row sep=1.5cm]
1 \arrow{r} &
G \arrow{r} &
(F^{c}G)^{\times} \arrow{r} &
(F^{c}G)^{\times}/G \arrow{r}&
1
\end{tikzcd}
\end{equation}
Taking $\Omega_F$-cohomology then yields the exact sequence
\[
\begin{tikzcd}[column sep=1cm, row sep=1.5cm]
1 \arrow{r} &
G \arrow{r} &
(F^{c}G)^{\times} \arrow{r} &
\mathcal{H}(FG) \arrow{r}{\delta}&
\mbox{Hom}(\Omega_F,G) \arrow{r}&
1,
\end{tikzcd}
\]
where $H^1(\Omega_F,(F^cG)^\times)=1$ by Hilbert's Theorem 90. Alternatively, for any $h\in\mbox{Hom}(\Omega_F,G)$ and $\textbf{r}_{G}(a)G\in\mathcal{H}(FG)$, we have $\delta(\textbf{r}_G(a))=h$ if and only if
\[
h(\omega)=\textbf{r}_G(a)^{-1}(\omega\cdot\textbf{r}_{G}(a))
\hspace{1cm}\mbox{for all }\omega\in\Omega_F,
\]
which is equivalent to $F_{h}=FG\cdot a$ by (\ref{eq:2.1}) and Proposition~\ref{prop:2.6}. Such an $a$ always exists by the Normal Basis Theorem. This shows that $\delta$ is surjective.

The argument above also shows that
\begin{equation}\label{eq:6.1'}
\mathcal{H}(FG)=\{\textbf{r}_{G}(a)G\mid F_{h}=FG\cdot a\mbox{ for }h\in\mbox{Hom}(\Omega_F,G)\}.
\end{equation}

\begin{definition}\label{def:6.1}Let $a\in\mbox{Map}(G,F^c)$ be such that $\textbf{r}_G(a)\in (F^cG)^\times$. Define
\[
r_G(a):=\textbf{r}_G(a)G,
\]
called the \emph{reduced resolvend of $a$}. Furthermore, if $r_G(a)\in\mathcal{H}(FG)$, define
\[
h_a\in\mbox{Hom}(\Omega_F,G);\hspace{1em}h_a(\omega):=\textbf{r}_G(a)^{-1}(\omega\cdot\textbf{r}_{G}(a)),
\]
called the \emph{homomorphism associated to $r_G(a)$}. Note that the definition of $h_a$ is independent of the choice of the representative $\textbf{r}_G(a)$, and we have $F_h=FG\cdot a$ by (\ref{eq:2.1}) and Proposition~\ref{prop:2.6}. Finally, we say that $r_G(a)$ is \emph{unramified} if $h_a$ is unramified. Similarly for \emph{tame}, \emph{weakly ramified}, and \emph{wildly ramified}.
\end{definition}

We similarly define
\[
\mathcal{H}(\mathcal{O}_FG):=((\mathcal{O}_{F^c}G)^\times/G)^{\Omega_F}.
\]
However, the corresponding analogue of sequence (\ref{eq:6.1}) will no ne longer exact on the right. Nevertheless, for $F$ a finite extension of $\mathbb{Q}_p$, a similar argument as above together with Corollary~\ref{cor:2.10} shows that
\begin{equation}\label{eq:6.1''}
\mathcal{H}(\mathcal{O}_FG)=\{r_G(a)\mid \mathcal{O}_h=\mathcal{O}_FG\cdot a\mbox{ for unramified $h\in\mbox{Hom}(\Omega_F,G)$}\}.
\end{equation}

\begin{definition}\label{def:6.2}For $F$ a number field, define $J(\mathcal{H}(FG))$ to be the restricted direct product of $\mathcal{H}(F_vG)$ with respect to the subgroups $\mathcal{H}(\mathcal{O}_{F_v}G)$ as $v$ ranges over $M_F$. Moreover, let
\[
\eta=\eta_F:\mathcal{H}(FG)\longrightarrow J(\mathcal{H}(FG))
\]
be the diagonal map given by the chosen embeddings $i_v:F^c\longrightarrow F_v^c$ and
\[
U(\mathcal{H}(\mathcal{O}_FG)):=\prod_{v\in M_F}\mathcal{H}(\mathcal{O}_{F_v}G)
\]
the group of unite id\`{e}les. 
\end{definition}

Next, we want to interpret resolvends and reduced resolvends as functions on characters of $G$. First, to simply notation, write
\[
\hat{G}=\hat{G}_F:=\mbox{Hom}(G,(F^c)^\times)
\]
for the set of $F^c$-valued irreducible characters on $G$. Define
\[
\det=\det\nolimits_F:\mathbb{Z}\hat{G}\longrightarrow\hat{G};\hspace{1em}
\det\left(\sum_{\chi} n_{\chi}\chi\right):=\prod_{\chi}\chi^{n\chi}
\]
and let
\[
A_{\hat{G}}=A_{\hat{G}_F}:=\ker(\det).
\]
Applying the functor $\mbox{Hom}(-,(F^{c})^{\times})$ to the short exact sequence
\[
\begin{tikzcd}[column sep=1cm, row sep=1.5cm]
0 \arrow{r} &
A_{\hat{G}} \arrow{r} &
\mathbb{Z}\hat{G}\arrow{r}[font=\normalsize, auto]{\det} &
\hat{G} \arrow{r} &
1
\end{tikzcd}
\]
yields the short exact sequence
\begin{equation}\label{eq:6.2}
\begin{tikzcd}[column sep=0.4cm, row sep=1.5cm]
1 \arrow{r} &
\mbox{Hom}(\hat{G},(F^{c})^{\times}) \arrow{r} &
\mbox{Hom}(\mathbb{Z}\hat{G},(F^{c})^{\times}) \arrow{r}&
\mbox{Hom}(A_{\hat{G}},(F^{c})^{\times}) \arrow{r}&
1,
\end{tikzcd}
\end{equation}
where exactness on the right follows from the fact that $(F^{c})^{\times}$ is divisible and thus injective. We will identify (\ref{eq:6.2}) with (\ref{eq:6.1}) as follows.

First, we have a canonical identification
\[
G=\hat{\hat{G}}=\mbox{Hom}(\hat{G},(F^{c})^{\times}).
\]
Secondly, we have canonical identifications
\begin{align}\label{eq:6.3}
(F^{c}G)^{\times}&=\mbox{Map}(\hat{G},(F^{c})^{\times})\\\notag
&=\mbox{Hom}(\mathbb{Z}\hat{G},(F^{c})^{\times}),\notag
\end{align}
where the second identification is obtained by extending maps $\hat{G}\longrightarrow(F^c)^\times$ via $\mathbb{Z}$-linearity, and the first is induced by characters in the following manner. Each $\textbf{r}_{G}(a)\in (F^cG)^\times$ is associated to the map $\varphi\in\mbox{Map}(\hat{G},(F^c)^\times)$ given
\begin{equation}\label{eq:6.4}
\varphi(\chi):=\sum_{s\in G}a(s)\chi(s)^{-1}.
\end{equation}
Conversely, given $\varphi\in\mbox{Map}(\hat{G},(F^c)^\times)$ one recovers  $\textbf{r}_{G}(a)$ by the formula
\begin{equation}\label{eq:6.5}
a(s):=\frac{1}{|G|}\sum_{\chi}\varphi(\chi)\chi(s).
\end{equation}
The third terms
\[
(F^{c}G)^{\times}/G=\mbox{Hom}(A_{\hat{G}},(F^{c})^{\times})
\]
in (\ref{eq:6.1}) and (\ref{eq:6.2}), respectively, are then canoncially identified.

From these identifications, we obtain a commutative diagram
\begin{equation}\label{eq:6.6}
\begin{tikzpicture}[baseline=(current bounding box.center)]
\node at (0,2.5) [name=13] {$\mbox{Hom}(\mathbb{Z}\hat{G},(F^c)^\times)$};
\node at (6,2.5) [name=14] {$\mbox{Hom}(A_{\hat{G}},(F^c)^\times)$};
\node at (0,0) [name=23] {$(F^cG)^\times$};
\node at (6,0) [name=24] {$(F^cG)^\times/G$};
\path[->]
(13) edge node[auto] {$rag_F$} (14)
(23) edge node[below] {$rag_F$} (24);
\draw[double equal sign distance]
(13) -- (23)
(14) -- (24);
\end{tikzpicture}
\end{equation}
where $rag=rag_F$ at the top is restriction to $A_{\hat{G}}$ and the corresponding map at the bottom is the natural quotient map. Taking $\Omega_F$-invariants then yields the commutative diagram
\begin{equation}\label{eq:6.7}
\begin{tikzpicture}[baseline=(current bounding box.center)]
\node at (0,2.5) [name=13] {$\mbox{Hom}_{\Omega_F}(\mathbb{Z}\hat{G},(F^c)^\times)$};
\node at (6,2.5) [name=14] {$\mbox{Hom}_{\Omega_F}(A_{\hat{G}},(F^c)^\times)$};
\node at (0,0) [name=23] {$(FG)^\times$};
\node at (6,0) [name=24] {$\mathcal{H}(FG)$.};
\path[->]
(13) edge node[auto]{$rag_F$} (14)
(23) edge node[below]{$rag_F$} (24);
\draw[double equal sign distance]
(13) -- (23)
(14) -- (24);
\end{tikzpicture}
\end{equation}

\begin{definition}\label{def:6.3}For $F$ a number field, observe that
\begin{equation}\label{eq:6.8}
\prod_{v\in M_F}rag_{F_v}:J(FG)\longrightarrow J(\mathcal{H}(FG))
\end{equation}
is clearly well-defined and that the diagram
\[
\begin{tikzpicture}[baseline=(current bounding box.center)]
\node at (0,2.5) [name=13] {$(FG)^\times$};
\node at (5.5,2.5) [name=14] {$J(FG)$};
\node at (0,0) [name=23] {$\mathcal{H}(FG)$};
\node at (5.5,0) [name=24] {$J(\mathcal{H}(FG))$};
\path[->]
(13) edge node[auto]{$\partial$} (14)
(23) edge node[below]{$\eta$} (24)
(14) edge node[right]{$\prod_v rag_{F_v}$} (24)
(13) edge node[left]{$rag_F$} (23);
\end{tikzpicture}
\]
commutes. By abuse of notation, we write $rag=rag_F$ for the map (\ref{eq:6.8}).
\end{definition}

Via the identifications in (\ref{eq:6.6}) and (\ref{eq:6.7}), we have that
\begin{align*}
(\mathcal{O}_{F^c}G)^\times&\subset\mbox{Hom}(\mathbb{Z}\hat{G},\mathcal{O}_{F^c}^\times);\\
\mathcal{H}(\mathcal{O}_FG)&\subset\mbox{Hom}_{\Omega_F}(A_{\hat{G}},\mathcal{O}_{F^c}^\times)
\end{align*}
These two inclusions, however, are strict in general.

\begin{prop}\label{prop:6.4}For $F$ a finite extension of $\mathbb{Q}_p$ with $(p,|G|)=1$, we have
\begin{align*}
(\mathcal{O}_{F^c}G)^\times&=\mbox{Hom}(\mathbb{Z}\hat{G},\mathcal{O}_{F^c}^\times);\\
\mathcal{H}(\mathcal{O}_FG)&=\mbox{Hom}_{\Omega_F}(A_{\hat{G}},\mathcal{O}_{F^c}^\times).
\end{align*}
\end{prop}
\begin{proof}It is clear from (\ref{eq:6.4}) and (\ref{eq:6.5}) that
\[
|G|\cdot\mbox{Hom}(\mathbb{Z}\hat{G},\mathcal{O}_{F^c}^{\times})\subset (\mathcal{O}_{F^c}G)^{\times}\subset\mbox{Hom}(\mathbb{Z}\hat{G},\mathcal{O}_{F^c}^{\times}).
\]
But $|G|$ is a unit in $F$ because $p$ does not divide $|G|$. So in fact
\[
(\mathcal{O}_{F^c}G)^\times=\mbox{Hom}(\mathbb{Z}\hat{G},\mathcal{O}_{F^c}^\times),
\]
proving the first equality. The second equality follows from the first.
\end{proof}


\section{The Modified Stickelberger Transpose}\label{s:7}

Let $F$ be a number field or finite extension of $\mathbb{Q}_p$. By modifying what has already been done in \cite[Section 4]{M}, we define a modified Stickelberger map. For $F$ a finite extension of $\mathbb{Q}_p$, the transpose of this map will play a key role in the characterization of reduced resolvends $r_G(a)$, where $A_h=\mathcal{O}_FG\cdot a$ for some weakly ramified $h\in\mbox{Hom}(\Omega_F,G)$ such that $F^h/F$ is totally ramified.

Throughout this section, we let 
\[
\{\zeta_n=\zeta_{n,F}:n\in\mathbb{Z}^+\}
\]
be the set of chosen primitive roots of unity in $F^c$. Moreover, as in Section~\ref{s:6}, we write $\hat{G}$ for the group $\hat{G}_F$ of $F^c$-valued irreducible characters on $G$.

\begin{definition}\label{def:7.1}For each $\chi\in\hat{G}$ and $s\in G$, let
\[
\upsilon(\chi,s)\in \left[\frac{1-|s|}{2},\frac{|s|-1}{2}\right]
\]
denote the unique integer (recall that $G$ has odd order) such that
\[
\chi(s)=(\zeta_{|s|})^{\upsilon(\chi,s)}
\]
and define
\[
\langle\chi,s\rangle_{*}:=\upsilon(\chi,s)/|s|.
\]
Extending this definition by $\mathbb{Q}$-linearity, we obtain a pairing
\[
\langle\hspace{1mm},\hspace{1mm}\rangle_*=\langle\hspace{1mm},\hspace{1mm}\rangle_{*,F}:\mathbb{Q}\hat{G}\times\mathbb{Q}G\longrightarrow\mathbb{Q},
\]
called the \emph{modified Stickelberger pairing}. The map
\[
\Theta_{*}=\Theta_{*F}:\mathbb{Q}\hat{G}\longrightarrow\mathbb{Q}G;
\hspace{1em}
\Theta_{*}(\psi):=\sum_{s\in G}\langle\psi,s\rangle_{*}s
\]
is called the \emph{modified Stickelberger map}.
\end{definition}

\begin{prop}\label{prop:7.2}
Let $\psi\in\mathbb{Z}\hat{G}$. Then $\Theta_{*}(\psi)\in\mathbb{Z}G$ if and only if $\psi\in A_{\hat{G}}$.
\end{prop}
\begin{proof}
Write $\psi=\sum_\chi n_{\chi}\chi$ with $n_\chi\in\mathbb{Z}$. For any $s\in G$, we have 
\begin{align*}
(\det\psi)(s)&=\prod_{\chi\in\hat{G}}\chi(s)^{n_{\chi}}\\
&=\prod_{\chi\in\hat{G}} (\zeta_{|s|})^{\upsilon(\chi,s)n_\chi}\\
&=(\zeta_{|s|})^{\sum_\chi|s|\langle\chi,s\rangle_*n_\chi}\\
&=(\zeta_{|s|})^{|s|\langle\psi,s\rangle_*}
\end{align*}
Since $A_{\hat{G}}=\ker(\det)$, this implies that $\psi\in A_{\hat{G}}$ precisely when $\langle\psi,s\rangle_*\in\mathbb{Z}$ for all $s\in G$, or equivalently, when $\Theta_*(\psi)\in\mathbb{Z}G$. This proves the claim.
\end{proof}

Up until now, we have let $\Omega_F$ act trivially on $G$. Below we introduce other $\Omega_F$-actions on $G$, one of which will make the $F$-linear map $\Theta_*:\mathbb{Q}\hat{G}\longrightarrow\mathbb{Q}G$ preserve $\Omega_F$-action. Here, the $\Omega_F$-action on $\hat{G}$ is the canonical one induced by the $\Omega_F$-action on the roots of unity.

\begin{definition}\label{def:7.3}
Let $m=\exp(G)$ and
\[
\kappa:\Omega_F\longrightarrow(\mathbb{Z}/m\mathbb{Z})^{\times}
\]
the $m$\emph{-th cyclotomic character of $\Omega_F$}. In other words, if $\mu_m$ denotes the group of $m$-th roots of unity in $F^c$, then 
\[
\omega(\zeta)=\zeta^{\kappa(\omega)}\hspace{1cm}\mbox{for }\zeta\in\mu_m\mbox{ and }\omega\in\Omega_F.
\]
For $n\in\mathbb{Z}$, let $G(n)$ be the group $G$ equipped with the $\Omega_F$-action given by
\[
\omega\cdot s:=s^{\kappa(\omega^{n})}\hspace{1cm}\mbox{for $s\in G$ and $\omega\in\Omega_F$}.
\]
\end{definition}

\begin{prop}\label{prop:7.4}
The map $\Theta_{*}:\mathbb{Q}\hat{G}\longrightarrow\mathbb{Q}G(-1)$ preserves $\Omega_F$-action.
\end{prop}
\begin{proof}
Observe that for any $\chi\in\hat{G}$ and $s\in G(-1)$, we have
\[
(\omega\cdot\chi)(s)=\chi(s^{\kappa(\omega)})=\chi(\omega^{-1}\cdot s).
\]
Since $s$ and $\omega^{-1}\cdot s$ have the same order, this implies that
\[
\langle\omega\cdot\chi,s\rangle_{*}=\langle\chi,\omega^{-1}\cdot s\rangle_{*}
\]
We deduce that
\begin{align*}
\Theta_{*}(\omega\cdot\chi)&=\sum_{s\in G}\langle\omega\cdot\chi,s\rangle_{*}s\\
&=\sum_{s\in G}\langle\chi,\omega^{-1}\cdot s\rangle_{*}s\\
&=\sum_{s\in G}\langle\chi,s\rangle_{*} (\omega\cdot s)\\
&=\omega\cdot\Theta_{*}(\chi).
\end{align*}
Since $\Omega_F$ acts trivially on $\mathbb{Q}$, this shows that $\Theta_{*}$ preserves $\Omega$-action.
\end{proof}

From Propositions~\ref{prop:7.2} and \ref{prop:7.4}, we obtain an $\Omega_F$-equivariant map
\[
\Theta_{*}:A_{\hat{G}}\longrightarrow\mathbb{Z}G(-1)
\]
Applying $\mbox{Hom}(-,(F^c)^\times)$ then yields an $\Omega_F$-equivariant homomorphism
\[
\Theta_{*}^{t}:\mbox{Hom}(\mathbb{Z}G(-1),(F^{c})^{\times})\longrightarrow\mbox{Hom}(A_{\hat{G}},(F^{c})^{\times});\hspace{1em}f\mapsto f\circ\Theta_*,
\]
where $\Omega_F$ acts on homomorphisms as usual: if $X$ and $Y$ are left $\Omega_F$-modules and $\varphi:X\longrightarrow Y$ is a group homomorphism, then
\[
(\varphi\cdot\omega)(x):=\omega^{-1}\cdot\varphi(\omega\cdot x)\hspace{1cm}\mbox{for $x\in X$ and $\omega\in\Omega_F$}.
\]
Restricting to elements which are $\Omega_F$-invariant, we obtain a homomorphism
\[
\Theta^t_{*}=\Theta_{*,F}^{t}:\mbox{Hom}_{\Omega_F}(\mathbb{Z}G(-1),(F^{c})^{\times})\longrightarrow\mbox{Hom}_{\Omega_F}(A_{\hat{G}},(F^{c})^{\times}),
\]
called the \emph{modified Stickelberger transpose}.

To simplify notation, we let 
\begin{align}\label{eq:7.1}
\Lambda(FG)&:=\mbox{Map}_{\Omega_F}(G(-1),F^c);\\\notag
\Lambda(\mathcal{O}_FG)&:=\mbox{Map}_{\Omega_F}(G(-1),\mathcal{O}_{F^c}).
\end{align}
Identifying $\mbox{Hom}_{\Omega_F}(A_{\hat{G}},(F^c)^\times)$ with $\mathcal{H}(FG)$ via (\ref{eq:6.7}), we see that $\Theta_{*}^{t}$ may be viewed as a homomorphism $\Lambda(FG)^{\times}\longrightarrow\mathcal{H}(FG)$ via (\ref{eq:6.7}).
 Moreover, notice that  we have
\begin{equation}\label{eq:7.2}
\Theta^t_*(\Lambda(\mathcal{O}_FG)^\times)\subset\mbox{Hom}_{\Omega_F}(A_{\hat{G}},\mathcal{O}_{F^c}^\times).
\end{equation}

\begin{prop}\label{prop:7.5}Let $g\in\Lambda(FG)^\times$ and $\Theta^t_*(g)=r_G(a)$. Then 
\[
\textbf{r}_G(a)\textbf{r}_G(a)^{[-1]}=1.
\]
\end{prop}
\begin{proof}View $\textbf{r}_G(a)$ as an element of $\mbox{Hom}(\mathbb{Z}\hat{G},(F^c)^\times)$ using (\ref{eq:6.3}). Let $\psi\in A_{\hat{G}}$ be given. Write $\psi=\sum_\chi n_\chi\chi\in A_{\hat{G}}$ with $n_\chi\in\mathbb{Z}$ and define
\[
\psi^{[-1]}:=\sum_{\chi\in\hat{G}}n_\chi\chi^{-1}.
\]
Then, it follows directly from (\ref{eq:6.4}) that 
\[
\textbf{r}_G(a)^{[-1]}(\psi)=\textbf{r}_G(a)(\psi^{[-1]}).
\]
Observe that $\langle\chi^{-1},s\rangle_*=-\langle\chi,s\rangle_*$ for $\chi\in\hat{G}$ and $s\in G$. We easily see that
\[
\Theta^t_*(g)(\psi^{[-1]})=\Theta^t_*(g)(\psi)^{-1}.
\]
Since $r_G(a)$ is simply the restriction of $\textbf{r}_G(a)$ to $A_{\hat{G}}$ via (\ref{eq:6.6}), we deduce that
\begin{align*}
(\textbf{r}_G(a)\textbf{r}_G(a)^{[-1]})(\psi)
&=\textbf{r}_G(a)(\psi)\textbf{r}_G(a)(\psi^{[-1})\\
&=\Theta^t_*(g)(\psi)\Theta^t_*(g)(\psi^{[-1]})\\
&=1.
\end{align*}
This implies that $\textbf{r}_G(a)\textbf{r}_G(a)^{[-1]}$ is the trivial map when restricted to $A_{\hat{G}}$. By (\ref{eq:6.6}), this means that there exists $t\in G$ such that
\[
\textbf{r}_G(a)\textbf{r}_G(a)^{[-1]}=t.
\]
Applying $[-1]$ to the above then yields $\textbf{r}_G(a)^{[-1]}\textbf{r}_G(a)=t^{-1}$. But $G$ has odd order, so we must have $t=1$, which proves the claim.
\end{proof}

\begin{definition}\label{def:7.6}
For $F$ a number field, define $J(\Lambda(FG))$ to be the restricted direct product of $\Lambda(F_vG)^\times$ with respect to the subgroups $\Lambda(\mathcal{O}_{F_v}G)^\times$ as $v$ ranges over $M_F$. Let
\[
\lambda=\lambda_F:\Lambda(FG)^\times\longrightarrow J(\Lambda(FG))
\]
be the diagonal map induced by the chosen embeddings $i_v:F^c\longrightarrow F_v^c$ and
\[
U(\Lambda(\mathcal{O}_FG)):=\prod_{v\in M_F}\Lambda(\mathcal{O}_{F_v}G)^\times
\]
the group of unit id\`{e}les. Observe that
\begin{equation}\label{eq:7.4}
\prod_{v\in F_F}\Theta^t_{*,F_v}:J(\Lambda(FG))\longrightarrow J(\mathcal{H}(FG))
\end{equation}
is well-defined by (\ref{eq:7.2}) and Proposition~\ref{prop:6.4}. Moreover, the diagram
\[
\begin{tikzpicture}[baseline=(current bounding box.center)]
\node at (0,2.5) [name=13] {$\Lambda(FG)^\times$};
\node at (5.5,2.5) [name=14] {$J(\Lambda(FG))$};
\node at (0,0) [name=23] {$\mathcal{H}(FG)$};
\node at (5.5,0) [name=24] {$J(\mathcal{H}(FG))$};
\path[->, font=\small]
(13) edge node[auto]{$\lambda$} (14)
(23) edge node[below]{$\eta$} (24)
(14) edge node[right]{$\prod_v\Theta^t_{*,F_v}$} (24)
(13) edge node[left]{$\Theta^t_{*,F}$} (23);
\end{tikzpicture}
\]
commutes because our convention is to choose
\[
\{\zeta_{n,F_v}:=i_v(\zeta_{n,F}):n\in\mathbb{Z}^+\}
\]
to be the set of distinguished primitive roots of unity in $F_v^c$ for each $v\in M_F$. By abuse of notation, we will also write $\Theta^t_*=\Theta^t_{*,F}$ for the map in (\ref{eq:7.4}).
\end{definition}


\section{Decomposition of Local Tame Resolvends}\label{s:8}

In this section, unless specified $F$ will always denote a finite extension of $\mathbb{Q}_p$. Modifying what has already been done in \cite[Section 5]{M}, we characterize reduced resolvends $r_G(a)$, where $A_h=\mathcal{O}_FG\cdot a$ for a tame $h\in\mbox{Hom}(\Omega_F,G)$. 

Let $\pi:=\pi_F$ be the chosen uniformizer of $F$ and $q:=q_F$, which denotes the order of the residue field $\mathcal{O}_F/(\pi)$. Moreover, let 
\[
\{\zeta_n=\zeta_{n,F}:n\in\mathbb{Z}^+\}
\]
be the chosen set of distinguished primitive roots of unity in $F^c$, and $F^{nr}$ the maximal unramified extension of $F$ contained in $F^c$, and

The structures of $F^{nr}/F$ and $F^{t}/F$ are well-known (see \cite[Sections 7 and 8]{F-CF}, for example). On one hand, the field $F^{nr}$ is obtained by adjoining to $F$ all $n$-th roots of unity for $(n,p)=1$, and so $\mbox{Gal}(F^{nr}/F)$ is a procyclic group topologically generated by the Frobenius automorphism $\phi=\phi_F$ given by
\[
\phi(\zeta_n)=\zeta_n^q\hspace{1cm}\mbox{for all }(n,p)=1.
\]
On the other hand, the field $F^t$ is obtained by adjoining to $F^{nr}$ all $n$-th roots of $\pi$ for $(n,p)=1$. Choose a coherent set of radicals 
\[
\{\pi^{1/n}:n\in\mathbb{Z}^+\}
\] 
such that $(\pi^{1/mn})^{n}=\pi^{1/m}$ and then define $\pi^{m/n}:=(\pi^{1/n})^{m}$ for all $m,n\in\mathbb{Z}^+$. We then obtain a distinguished topological generator $\sigma=\sigma_F$ of the procyclic group $\mbox{Gal}(F^{t}/F^{nr})$, which is given by
\[
\sigma(\pi^{1/n})=\zeta_n\pi^{1/n}\hspace{1cm}\mbox{for all }(n,p)=1.
\]
If $\phi$ also denotes the unique lifting of $\phi$ from $\mbox{Gal}(F^{nr}/F)$ to $\Omega_F^{t}$ fixing the radicals $\pi^{1/n}$ for $(n,p)=1$, then $\Omega_F^{t}$ is topologically generated by $\phi$ and $\sigma$.
\begin{definition}\label{def:8.1}Let $h\in\mbox{Hom}(\Omega_F^t,G)$. Define
\begin{align*}
h^{nr}\in\mbox{Hom}(\Omega^t_F,G);\hspace{1em}&h^{nr}(\phi):=h(\phi)\mbox{ and }h^{nr}(\sigma):=1;\\
h^{tot}\in\mbox{Hom}(\Omega^t_F,G);\hspace{1em}&h^{tot}(\phi):=1\mbox{ and }h^{tot}(\sigma):=h(\sigma).
\end{align*}
Notice that $h^{nr}$ is clearly unramified and that $h=h^{nr}h^{tot}$, called the \emph{factorization of $h$ with respect to $\sigma$}.
\end{definition}

Observe that $\phi\sigma\phi^{-1}\sigma^{-1}=\sigma^{q-1}$ because the elements on both sides have the same effect on $\zeta_{n}$ and $\pi^{1/n}$ for $(n,p)=1$. So, the abelianization $(\Omega_F^{t})^{ab}$ of $\Omega_F^t$ is the direct product of the procyclic group topologically generated by $\overline{\phi}$ with the cyclic group $\langle\overline{\sigma}\rangle$ of order $q-1$, where $\overline{\sigma}$ and $\overline{\phi}$ are the images of $\phi$ and $\sigma$ in $(\Omega_F^{t})^{ab}$, respectively. Since $G$ is abelian, every $h\in\mbox{Hom}(\Omega_F^t,G)$ factorizes through $\Omega_F^t\longrightarrow(\Omega_F^t)^{ab}$, and $h(\sigma)$ is of order dividing $q-1$. Conversely, it is clear that $h(\sigma)$ can be any element in $G$ of order dividing $q-1$.

\begin{definition}\label{def:8.2}
Let $G_{(q-1)}$ be the subgroup of $G$ consisting of all elements of order dividing $q-1$. For each $s\in G_{(q-1)}$, define
\[
f_s=f_{F,s}\in\Lambda(FG);
\hspace{1em}f_{s}(t):=\begin{cases}
\pi & \mbox{if }t=s\neq1\\
1 & \mbox{otherwise}
\end{cases}
\]
(recall (\ref{eq:7.1})). Note that $f_{s}$ preserves $\Omega_F$-action because $F$ contains all $(q-1)$-st roots of unity, whence elements in $G_{(q-1)}$ are fixed by $\Omega_F$, as is $\pi$. Set
\[
\mathfrak{F}_F:=\{f_s:s\in G_{(q-1)}\}.
\]
\end{definition}

\begin{definition}\label{def:8.3}For $F$ a number field, define
\[
\mathfrak{F}=\mathfrak{F}_F:=\{f\in J(\Lambda(FG))\mid f_v\in\mathfrak{F}_{F_v}\mbox{ for all }v\in M_F\},
\]
whose elements are called \emph{prime $\mathfrak{F}$-elements}.
\end{definition}

\begin{prop}\label{prop:8.4}Let $s\in G_{(q-1)}$ and define
\[
h\in\mbox{Hom}(\Omega_F^{t},G);\hspace{1em}h(\phi)=1\mbox{ and }h(\sigma)=s.
\]
Then $F^{h}=F(\pi^{1/|s|})$, and there exists $a\in F_h$ such that $A_{h}=\mathcal{O}_FG\cdot a$ and 
\[
r_{G}(a)=\Theta_{*}^{t}(f_s).
\]
\end{prop}
\begin{proof}If $s=1$, then take $a\in F_h$ to be such that $a(t)=0$ if $t\neq 1$ and $a(1)=1$. Since $h=1$, it is clear that $A_h=\mathcal{O}_h=\mathcal{O}_FG\cdot a$ and $r_G(a)=1=\Theta^{t}_{*}(f_s)$.

Now assume that $s\neq 1$. Let $e:=|s|$ and $\Pi:=\pi^{1/e}$. Then $F^{h}=F(\Pi)$ since $\ker(h)$ is topologically generated by $\phi$ and $\sigma^{e}$, which both fix $\Pi$, and 
\[
[\Omega_F^t:\ker(h)]=e=[F(\Pi):F].
\]
Note that= $F^h/F$ is totally ramified and $\Pi$ is a uniformizer of $F^h$. It follows that $\mathcal{O}^h=\mathcal{O}_F[\Pi]$ (see \cite[Chapte I Proposition 18]{S}, for example). Since
\[
A^h=\Pi^{(1-e)/2}\mathcal{O}^h
\]
by Proposition~\ref{prop:1.1}, we see that
\begin{equation}\label{eq:8.1}
\{\Pi^{k+(1-e)/2}\mid k=0,1,\dots,e-1\}
\end{equation}
is an $\mathcal{O}_F$-basis for $A^{h}$. First, we show that $A^{h}=\mathcal{O}_F\mbox{Gal}(F^h/F)\cdot\alpha$, where
\[
\alpha:=\frac{1}{e}\sum_{k=0}^{e-1}\Pi^{k+\frac{1-e}{2}}.
\]
Note that $\alpha\in A^h$ since $(e,p)=1$ implies that $e\in\mathcal{O}_F^\times$.

Notice that $\mbox{Gal}(F^h/F)$ is a cyclic group of order $e$ generated by the restriction of $\sigma$ to $F^h$. For each $i=0,1,\dots,e-1$, we have
\[
\sigma^{i}(\alpha)=\frac{1}{e}\sum_{k=0}^{e-1}\zeta_{e}^{(k+\frac{1-e}{2})i}\Pi^{k+\frac{1-e}{2}}.
\]
Multiplying both sides by $\zeta_{e}^{-(l+(1-e)/2)i}$ yields
\[
\sigma^{i}(\alpha)\zeta_{e}^{-(l+\frac{(1-e)}{2})i}=\frac{1}{e}\sum_{k=0}^{e-1}\zeta_{e}^{(k-l)i}\Pi^{k+\frac{1-e}{2}}.
\]
Summing the above over all $l=0,1,\dots,e-1$, we obtain
\begin{equation}\label{eq:8.2}
\sum_{i=0}^{e-1}\sigma^{i}(\alpha)\zeta_{e}^{-(l+\frac{1-e}{2})i}
= \frac{1}{e}\sum_{k=0}^{e-1}\Pi^{k+\frac{1-e}{2}}\sum_{i=0}^{e-1}\zeta_{e}^{(k-l)i}
=\Pi^{l+\frac{1-e}{2}}.
\end{equation}
This shows that $A^h=\mathcal{O}_F\mbox{Gal}(F^h/F)\cdot\alpha$ since (\ref{eq:8.1}) is an $\mathcal{O}_F$-basis for $A^{h}$ and $\zeta_{e}\in\mathcal{O}_F$. Since $A_h=\mbox{Map}_{\Omega_F}(^hG,A^h)$, one easily verifies that $a\in\mbox{Map}(G,F^c)$ defined by
\[
a(t):=\begin{cases}
\omega(\alpha) & \mbox{if $t=h(\omega)$ for $\omega\in\Omega_F^t$}\\
0 & \mbox{otherwise}
\end{cases}
\]
is an element in $F_h$ and that $A_h=\mathcal{O}_FG\cdot a$. 

To show that $r_G(a)=\Theta^t_*(f_s)$, first observe that
\[
\textbf{r}_{G}(a)
=\sum_{t\in h(\Omega_F)}a(t)t^{-1}\\
=\sum_{i=0}^{e-1}\sigma^i(\alpha) s^{-i}.
\]
Next, let $\chi\in\hat{G}_F$ and $\upsilon:=\upsilon(\chi,s)$ be as in Definition~\ref{def:7.1}, and set
\[
k:=\upsilon-\frac{1-e}{2}\in\{0,1,\cdots,e-1\}.
\]
Then, we have
\[
\textbf{r}_{G}(a)(\chi)
=\sum_{i=0}^{e-1}\sigma^{i}(\alpha)\zeta_{e}^{-\upsilon  i}
=\sum_{i=0}^{e-1}\sigma^{i}(\alpha)\zeta_{e}^{-(k+\frac{1-e}{2}) i}
\]
and the same computation as in (\ref{eq:8.2}) shows that
\[
\textbf{r}_{G}(a)(\chi)=\Pi^{k+\frac{1-e}{2}}=\pi^{\langle\chi,s\rangle_{*}}.
\]
On the other hand, we have
\[
\Theta_{*}^{t}(f_{s})(\chi)
=f_s\Big(\sum_{t\in G}\langle\chi,t\rangle_*t\Big)\\
=\prod_{t\in G}f_{s}(t)^{\langle\chi,t\rangle_{*}}\\
=\pi^{\langle\chi,s\rangle_{*}}
\]
also. We conclude from the identification (\ref{eq:6.6}) that $r_{G}(a)=\Theta_{*}^{t}(f_{s})$.
\end{proof}

Next we consider an arbitrary $h\in\mbox{Hom}(\Omega_F^t,G)$.

\begin{thm}\label{thm:8.5}Let $h\in\mbox{Hom}(\Omega_F^{t},G)$. If $A_h=\mathcal{O}_FG\cdot a$, then we have
\[
r_{G}(a)=u\Theta_{*}^{t}(f_s)
\]
for some $u \in\mathcal{H}(\mathcal{O}_FG)$ and for $s:=h(\sigma)$.
\end{thm}
\begin{proof}Let $h=h^{nr}h^{tot}$ be the factorization of $h$ with respect to $\sigma$. By Corollary~\ref{cor:2.10} and (\ref{eq:6.1''}), there exists $a_{nr}\in F_{h^{nr}}$ such that $\mathcal{O}_{h^{nr}}=\mathcal{O}_FG\cdot a_{nr}$ and
\[
r_G(a_{nr})=u'\hspace{1cm}\mbox{for some }u'\in\mathcal{H}(\mathcal{O}_FG).
\]
By Proposition~\ref{prop:8.4}, there exists $a_{tot}\in F_{h^{tot}}$ such that $A_{h^{tot}}=\mathcal{O}_FG\cdot a_{tot}$ and
\[
r_G(a_{tot})=\Theta^t_s(f_s)\hspace{1cm}\mbox{for }s:=h(\sigma).
\]
Proposition~\ref{prop:4.4} then yields an element $a'\in F_h$ such that $A_h=\mathcal{O}_FG\cdot a'$ and
\[
\textbf{r}_{G}(a')=\textbf{r}_{G}(a_{nr})\textbf{r}_{G}(a_{tot}).
\]
But $a=\gamma\cdot a'$ for some $\gamma\in(\mathcal{O}_FG)^{\times}$ because $A_h=\mathcal{O}_FG\cdot a$ also, so 
\[
r_{G}(a)=rag(\gamma)r_{G}(a')=(rag(\gamma) u')\Theta_{*}^{t}(f_{s}),
\]
where $u:=rag(\gamma) u'\in\mathcal{H}(\mathcal{O}_FG)$. This proves the claim.
\end{proof}

\begin{thm}\label{thm:8.6}Let $s\in G_{(q-1)}$ and $u\in\mathcal{H}(\mathcal{O}_FG)$. If $h$ is the homomorphism associated to $u\Theta^t_*(f_s)$, then $h(\sigma)=s$ and 
\[
r_{G}(a)=u\Theta_{*}^{t}(f_s)
\]
for some $a\in F_h$ such that $A_h=\mathcal{O}_FG\cdot a$.
\end{thm}
\begin{proof}Let $h^{nr}$ and $h^{tot}$ be the homomorphisms associated to $u$ and $\Theta^{t}_{*}(f_s)$, respectively. From (\ref{eq:6.1''}), we know that $\mathcal{O}_{h^{nr}}=\mathcal{O}_FG\cdot a_{nr}$ for some $a_{nr}\in F_{h^{nr}}$ satisfying $r_G(a_{nr})=u$, and that $h^{nr}$ is unramified so $h^{nr}(\sigma)=1$. On the other hand, Proposition~\ref{prop:8.4} implies that $A_{h^{tot}}=\mathcal{O}_FG\cdot a_{tot}$ for some $a_{tot}\in F_{h^{tot}}$ such that $r_G(a_{tot})=\Theta^t_*(f_s)$, and that $h^{tot}(\sigma)=s$. Proposition~\ref{prop:4.4} then yields the desired element $a\in F_h$. Finally, we have $h(\sigma)=h^{nr}(\sigma)h^{tot}(\sigma)=s$.
\end{proof} 

Theorem~\ref{thm:8.7} may be viewed as a global version of Theorems~\ref{thm:8.5} and~\ref{thm:8.6}.

\begin{thm}\label{thm:8.7}Let $F$ be a number field. Let $h\in\mbox{Hom}(\Omega_F,G)$ and $F_h=FG\cdot b$. Then, we have $h$ is tame if and only if 
\begin{equation}\label{eq:8.3}
rag(c)=\eta(r_G(b))^{-1}u\Theta^t_*(f)
\end{equation}
for some $c\in J(FG)$, $u\in U(\mathcal{H}(\mathcal{O}_FG))$, and $f\in\mathfrak{F}$. Moreover, if this is the case, then $j(c)=\mbox{cl}(A_{h})$ and $f_v=f_{v,s_v}$, where $s_v=h_v(\sigma_{F_v})$, for all $v\in M_F$. In particular, we have $f_v=1$ if and only if $h_v$ is unramified. 
\end{thm}
\begin{proof}First assume that $h$ is tame. Let $F_h=FG\cdot b$ and $A_{h_v}=\mathcal{O}_{F_v}G\cdot a_v$ for $v\in M_F$ as in Remark~\ref{rem:3.5}. Moreover, let $c\in J(FG)$ be such that $a_v=c_v\cdot b$ for each $v\in M_F$, so $j(c)=\mbox{cl}(A_h)$. At each $v\in M_F$, we have
\[
r_G(a_v)=rag(c_v)r_G(b).
\]
Since $h_v$ is tame, we know from Theorem~\ref{thm:8.5} that 
\[
r_G(a_v)=u_v\Theta_*^t(f_{s_v})
\]
for some $u_v\in\mathcal{H}(\mathcal{O}_{F_v}G)$ and for $s_v:=h_v(\sigma_{F_v})$. Note that $f_{s_v}=1$ if and only if $s_v=1$, which occurs precisely when $h_v$ is unramified. This shows that $f_{s_v}=1$ for almost all $v\in M_F$ and so $f:=(f_{s_v})\in\mathfrak{F}$. Letting $c:=(c_v)\in J(FG)$ and $u:=(u_v)\in U(\mathcal{H}(\mathcal{O}_FG))$, we see that (\ref{eq:8.3}) holds. 

Conversely, assume that (\ref{eq:8.3}) holds. Then, at each $v\in M_F$ we have
\[
rag(c_v)r_G(b)=u_v\Theta_*^t(f_v),
\]
with $f_v=f_{s_v}$ say. This implies that $h_v$ is the homomorphism associated to $u_v\Theta^t_*(f_{s_v})$, which is tame by (\ref{eq:6.1''}) and Proposition~\ref{prop:8.4}, whence $h$ is tame. The fact that $h_v(\sigma_{F_v})=s_v$ for all $v\in M_F$ follows from Theorem~\ref{thm:8.6}.

It remains to show that $j(c)=\mbox{cl}(A_h)$. Again it follows from Theorem~\ref{thm:8.6} that there exists $a_v\in F_{h_v}$ such that $A_{h_v}=\mathcal{O}_{F_v}G\cdot a_v$ and
\[
r_G(a_v)=u_v\Theta^t_*(f_{s_v}).
\]
This implies that $r_G(a_v)=rag(c_v)r_G(b)$, so there exists $t_v\in G$ such that
\[
\textbf{r}_G(a_v)=c_v\textbf{r}_G(b)t=\textbf{r}_G(c_vt_v\cdot b),
\]
Since $\textbf{r}_G$ is bijective, the above is equivalent to
\[
a_v=(c_vt_v)\cdot b.
\]
If $t:=(t_v)\in J(FG)$, then Proposition~\ref{prop:3.4} gives us
$\mbox{cl}(A_h)=j(ct)=j(c)$, as desired. This completes the proof of the theorem.
\end{proof}


\section{Approximation Theorems}\label{s:9}

Let $F$ be a number field. Theorem~\ref{thm:8.7} implies that for any $c\in J(FG)$, we have $j(c)$ is tame $A$-realizable if and only if
\begin{equation}\label{eq:9.1}
rag(c)\in\eta(\mathcal{H}(FG))U(\mathcal{H}(\mathcal{O}_FG))\Theta^t_*(\mathfrak{F})
\end{equation}
It is not clear from the above that the tame $A$-realizable classes form a subgroup of $\mbox{Cl}(\mathcal{O}_FG)$ because $\mathfrak{F}$ is only a set. Below we state two approximation theorems from \cite{M}, which will allow us to replace $\mathfrak{F}$ by $J(\Lambda(FG))$ in (\ref{eq:9.1}).

\begin{definition}\label{def:9.1}
Let $\mathfrak{m}$ be an ideal in $\mathcal{O}_F$. For each $v\in M_F$, let
\begin{align*}
U_{\mathfrak{m}}(\mathcal{O}_{F_v^c})&:=(1+\mathfrak{m}\mathcal{O}_{F_v^c})\cap(\mathcal{O}_{F_v^c})^{\times};\\
U'_{\mathfrak{m}}(\Lambda(\mathcal{O}_{F_v}G))&:=\{g_v\in\Lambda(\mathcal{O}_{F_v}G)^\times\mid g_v(s)\in U_\mathfrak{m}(\mathcal{O}_{F_v^c})\mbox{ for all }s\in G,s\neq 1\},
\end{align*}
Furthermore, set
\[
U'_\mathfrak{m}(\Lambda(\mathcal{O}_FG))
:=\left(\prod_{v\in M_F}U_\mathfrak{m}'(\Lambda(\mathcal{O}_{F_v}G))\right)\cap J(\Lambda(FG)).
\]
The \emph{modified ray class group mod $\mathfrak{m}$ of $\Lambda(FG)$} is defined to be
\[
\mbox{Cl}'_{\mathfrak{m}}(\Lambda(FG)):=\frac{J(\Lambda(FG))}{\lambda(\Lambda(FG)^{\times})U'_{\mathfrak{m}}(\Lambda(\mathcal{O}_FG))}.
\]
\end{definition}

\begin{definition}\label{def:9.2}For $g\in J(\Lambda(FG))$ and $s\in G$, we define
\[
g_s:=\prod_{v\in M_F}g_v(s)\in\prod_{v\in M_F}(F_v^c)^\times.
\]
Moreover, choose a set $S=S_F$ of representatives of the $\Omega_F$-orbits of $G(-1)$.
\end{definition}

\begin{thm}\label{thm:9.3}Let $g\in J(\Lambda(FG))$ and $T$ a finite subset of $M_F$. Then, there exists $f\in\mathfrak{F}$ such that $f_v=1$ for all $v\in T$ and 
\[
g\equiv f\hspace{1cm}(\mbox{mod }\lambda(\Lambda(FG)^\times)U_\mathfrak{m}'(\Lambda(\mathcal{O}_FG)))
\]
Moreover, the element $f$ can be chosen such that $f_s\neq 1$ for all $s\in S\backslash\{1\}$.
\end{thm}
\begin{proof}See \cite[Proposition 6.14]{M}.
\end{proof}

\begin{thm}\label{thm:9.4}Let $\mathfrak{m}$ be an ideal in $\mathcal{O}_F$ divisible by $|G|$ and $\exp(G)^2$. Then,
\[
\mbox{Hom}_{\Omega_{F_v}}(A_{\hat{G}},U_\mathfrak{m}(\mathcal{O}_{F_v^c})\subset\mathcal{H}(\mathcal{O}_{F_v}G)\hspace{1cm}
\mbox{ for all }v\in M_F.
\]
\end{thm}
\begin{proof}See \cite[Theorem 2.14]{M}.
\end{proof}

\begin{cor}\label{cor:9.5}Let $\mathfrak{m}$ be an ideal in $\mathcal{O}_F$ divisible by $|G|$ and $\exp(G)^2$. Then,
\[
\Theta_*^t(U_\mathfrak{m}'(\Lambda(\mathcal{O}_FG))\subset U(\mathcal{H}(\mathcal{O}_FG))\hspace{1cm}\mbox{for all }v\in M_F.
\]
\end{cor}
\begin{proof}Let $v\in M_F$ be given. By Theorem~\ref{thm:9.4}, it suffices to show that
\begin{equation}\label{eq:9.2}
\Theta^t_*(U_\mathfrak{m}'(\Lambda(\mathcal{O}_{F_v}G)))\subset
\mbox{Hom}_{\Omega_{F_v}}(A_{\hat{G}},U_\mathfrak{m}(\mathcal{O}_{F_v^c}))
\end{equation}
To that end, let $g_{v}\in U_\mathfrak{m}'(\Lambda(\mathcal{O}_{F_v}G))$. For any $\psi\in A_{\hat{G}}$, we have
\[
\Theta_{*}^{t}(g_{v})(\psi)
=g_{v}\left(\sum_{s\in G}\langle\psi,s\rangle_{*}s\right)
=\prod_{s\neq 1}g_v(s)^{\langle\psi,s\rangle}_*
\]
because $\langle\psi,1\rangle_*=0$ by Definition~\ref{def:7.1}. Moreover, since $\langle\psi,s\rangle_{*}\in\mathbb{Z}$ by Proposition~\ref{prop:7.2} and $g_{v}(s)\in U_{\mathfrak{m}}(\mathcal{O}_{F_v^c})$ by definition for each $s\neq1$, it is clear that $\Theta_{*}^{t}(g_{v})(\psi)\in U_{\mathfrak{m}}(\mathcal{O}_{F_v^c})$. This proves (\ref{eq:9.2}), as desired.
\end{proof}


\section{Proof of Theorem~\ref{thm:1.3}}\label{s:10}

\begin{proof}[Proof of Theorem~\ref{thm:1.3} (a)] 
Let $\upvarrho$ be the composite of the homomorphism $rag$ defined in Definition~\ref{def:6.3} followed by the natural quotient map
\[
J(\mathcal{H}(KG))\longrightarrow\frac{J(\mathcal{H}(KG))}{\eta(\mathcal{H}(KG))U(\mathcal{H}(\mathcal{O}_KG))\Theta^t_*(J(\Lambda(KG)))}.
\]
We show that $\mathcal{A}^t(\mathcal{O}_KG)$ is a subgroup of $\mbox{Cl}(\mathcal{O}_KG)$ by showing that
\[
j^{-1}(\mathcal{A}^t(\mathcal{O}_KG))=\ker(\upvarrho).
\]
In other words, given $c\in J(KG)$, we have $j(c)\in\mathcal{A}^t(\mathcal{O}_KG)$ if and only if
\begin{equation}\label{eq:10.1}
rag(c)\in\partial(\mathcal{H}(KG))U(\mathcal{H}(\mathcal{O}_KG))\Theta^t_*(J(\Lambda(FG))).
\end{equation}

The inclusion $j^{-1}(\mathcal{A}^t(\mathcal{O}_KG))\subset\ker(\upvarrho)$ follows from Theorem~\ref{thm:8.7} and (\ref{eq:6.1'}). To prove the other inclusion, let $c\in J(KG)$ be such that $j(c)\in\ker(\upvarrho)$. Then
\begin{equation}\label{eq:10.2}
rag(c)=\eta(r_G(b))^{-1}u\Theta^t_*(g)
\end{equation}
for some $r_G(b)\in\mathcal{H}(KG)$, $u\in U(\mathcal{H}(\mathcal{O}_KG))$, and $g\in J(\Lambda(KG))$. Let $\mathfrak{m}$ be an ideal in $\mathcal{O}_K$ divisible by $|G|$ and $\exp(G)^2$. Then, we obtain from Theorem~\ref{thm:9.3} an element $f\in\mathfrak{F}$ such that 
\[
g\equiv f\hspace{1cm}(\mbox{mod }\lambda(\Lambda(FG)^\times)U_\mathfrak{m}'(\Lambda(\mathcal{O}_FG)))
\]
By Corollary~\ref{cor:9.5}, applying $\Theta^t_*$ to the above yields
\[
\Theta^t_*(g)\equiv \Theta^t_*(f)\hspace{1cm}(\mbox{mod }\eta(\mathcal{H}(KG))U(\mathcal{H}(\mathcal{O}_FG)))
\]
Hence, changing $b$ and $u$ in (\ref{eq:10.2}) if necessary, we may assume that $g=f$. Letting $h:=h_b$, we see from Theorem~\ref{thm:8.7} that $h$ is tame and $j(c)=\mbox{cl}(A_h)$. This shows that $j(c)\in\mathcal{A}^t(\mathcal{O}_KG)$, which completes the proof of (\ref{eq:10.1}).
\end{proof}

\begin{proof}[Proof of Theorem~\ref{thm:1.3} (b)]
Let $c\in J(KG)$ be such that $j(c)\in\mathcal{A}^t(\mathcal{O}_KG)$. Then (\ref{eq:10.2}) holds by (\ref{eq:10.1}). By the same argument following (\ref{eq:10.2}), we may assume that $g=f$ lies in $\mathfrak{F}$, in which case $h:=h_b$ is tame and $j(c)=\mbox{cl}(A_h)$. By Theorem~\ref{thm:9.3}, we may in fact assume that
\begin{align}
\label{eq:10.3}f_v=1\hspace{1cm}&\mbox{for all }v\in T;\\
\label{eq:10.4}f_s\neq 1\hspace{1cm}&\mbox{for all }s\in S\backslash\{1\},
\end{align}
where $T$ is a given finite set of primes in $\mathcal{O}_K$. Theorem~\ref{thm:8.7} and (\ref{eq:10.3}) imply that $h_v$ is unramified for all $v\in T$. Recall from Definition~\ref{def:9.2} that $S$ a set of representatives of the $\Omega_K$-orbits of $G(-1)$, so (\ref{eq:10.4}) holds for all $s\in G\backslash\{1\}$. In particular, given any $s\in G$, there exists $v\in M_K$ such that $f_v=f_{F_v,s}$, whence $h_v(\sigma_{F_v})=s$ by Theorem~\ref{thm:8.7}. This shows that $h$ is surjective, whence $K_h$ is a field. This completes the proof of the theorem.
\end{proof}


\section{Decomposition of Local Homomorphisms}\label{s:11}

Let $F$ be a finite extension of $\mathbb{Q}_p$. We give a generalization of Definition~\ref{def:8.1} which applies to arbitrary and not necessarily tame $h\in\mbox{Hom}(\Omega_F,G)$. To that end, let $\pi=\pi_F$ be the chosen uniformizer of $F$. Moreover, for each $n\in\mathbb{Z}_{\geq 0}$, let $F_{\pi,n}$ denote the $n$-th Lubin-Tate division field of $F$ corresponding to $\pi$.

\begin{definition}\label{definition:11.1}
Let  $h\in\mbox{Hom}(\Omega_F,G)$. We say that
\[
h=h^{nr}h^{tot},\hspace{1cm}\mbox{where }h^{nr},h^{tot}\in\mbox{Hom}(\Omega_F,G)
\]
is a \emph{factorization of $h$ with respect to $\pi$} if $h^{nr}$ is unramified and $F^{h^{tot}}\subset F_{\pi,n}$ for some $n\in\mathbb{Z}_{\geq 0}$. The \emph{level} of such a decomposition is defined to be
\[
\ell_\pi(h^{nr}h^{tot}):=\min\{n\in\mathbb{Z}_{\geq 0}\mid F^{h^{tot}}\subset F_{\pi,n}\}.
\]
\end{definition}

\begin{remark}\label{rem:11.2}The factorization $h=h^{nr}h^{tot}$ of a tame $h\in\mbox{Hom}(\Omega^t_F,G)$ given in Definition~\ref{def:8.1} is a factorization with respect to $-\pi$. Moreover, note that a homomorphism $h\in\mbox{Hom}(\Omega_F,G)$ has a factorization of level $0$ if and only if $h$ is unramified, and a factorization of level $1$ if and only if $h$ is tame.
\end{remark}

\begin{prop}\label{prop:11.3}Every $h\in\mbox{Hom}(\Omega_F,G)$ has a factorization with respect to $\pi$. If $h$ is weakly ramified, then any such factorization has level at most $2$.
\end{prop}
\begin{proof}Let $F^{ab}$ denote the maximal abelian extension and $F^{nr}$ the maximal unramified extension of $F$ contained in $F^c$, respectively. Moreover, let $F_\pi$ be the union of all of the $F_{\pi,n}$ for $n\in\mathbb{Z}_{\geq 0}$. Then, fom Local Class Field Theory, we know that $F^{ab}=F^{nr}F_\pi$ and $F^{nr}\cap F_{\pi}=F$. Set $\Omega_F^{ab}:=\mbox{Gal}(F^{ab}/F)$. Then, there is a natural isomorphism
\[
\Omega_F^{ab}\simeq\mbox{Gal}(F^{nr}/F)\times\mbox{Gal}(F_\pi/F),
\]
and we may view $\mbox{Gal}(F^{nr}/F)$ and $\mbox{Gal}(F_\pi/F)$ as subgroups of $\Omega_F^{ab}$.

Now, since $G$ is abelian, every $h\in\mbox{Hom}(\Omega_F,G)$ may be viewed as a homomorphism $\Omega_F^{ab}\longrightarrow G$. Let $h^{nr}$ and $h^{tot}$ be the restrictions of $h$ to $\mbox{Gal}(F^{nr}/F)$ and $\mbox{Gal}(F_\pi/F)$, respectively. It is then clear that $h=h^{nr}h^{tot}$ is a factorization with respect to $\pi$.

Finally, if $h$ is weakly ramified and $h=h^{nr}h^{tot}$ is a factorization of $h$ with respect to $\pi$, then the fact that $\ell_\pi(h^{nr}h^{tot})\leq 2$ follows from the proofs of \cite[Propostion 4.1 and Lemma 4.2]{By}.
\end{proof}


\section{Construction of Local Wild Generators}\label{s:12}

Let $F$ be a finite extension of $\mathbb{Q}_p$. In this section, we will assume that $F/\mathbb{Q}_p$ is unramified and that $p$ is odd. Moreover, we choose $p$ to be the uniformizer of $F$, and $F_{\pi,n}$ is defined as in Section~\ref{s:11}. First we make an observation.

\begin{prop}\label{prop:12.1}Let $L/F$ be a finite Galois extension of ramification index $p$ and $\mathfrak{D}_{L/F}$ its different ideal. Under the assumptions that $F/\mathbb{Q}_p$ is unramified and that $p$ is odd, we have $L/F$ is weakly ramified and $v_{L}(\mathfrak{D}_{L/F})=2(p-1)$.
\end{prop}
\begin{proof}Let $G(L/F)_n$ denote the $n$-th ramification group of $L/F$. By hypothesis, we have $|G(L/F)_0|=p$. We also know that $|G(L/F)_1|\neq 1$ since $L/F$ is wildly ramified, whence $|G(L/F)_1|=p$ also. Now, suppose on the contrary that $|G(L/F)_2|\neq 1$. Then $|G(L/F)_2|=p$ and Proposition~\ref{prop:1.1} implies that
\[
v_L(\mathfrak{D}_{L/F})=\sum_{n=0}^{\infty}(G(L/F)_n-1)\geq 3(p-1).
\]
But we also know from \cite[Chapter III Theorem 2.5]{N} that
\[
v_L(\mathfrak{D}_{L/F})\leq p-1+v_L(p),
\]
where $v_L(p)=p$ since $F/\mathbb{Q}_p$ is unramified and $L/F$ has ramification index $p$. But then $2p-1\leq 3(p-1)$, so $p=2$ and this is a contradiction. This shows that $G(L/F)_2$ is trivial, which means that $L/F$ is weakly ramified, and the equality $v_{L}(\mathfrak{D}_{L/F})=2(p-1)$ follows from Proposition~\ref{prop:1.1}.
\end{proof}

In this rest of this section, we prove Proposition~\ref{prop:12.2} below, which is analogous to Proposition~\ref{prop:8.4}. Recall from (\ref{eq:7.1}) that $\Lambda(FG):=\mbox{Map}_{\Omega_F}(G(-1),F^c)$. 

\begin{prop}\label{prop:12.2}Let $h\in\mbox{Hom}(\Omega_F,G)$ be such that $F^h/F$ has ramification index $p$ and that $F^h\subset F_{p,2}$. Under the assumptions that $F/\mathbb{Q}_p$ is unramified and that $p$ is odd, there is $a\in F_h$ such that $A_h=\mathcal{O}_FG\cdot a$ and
\[
r_G(a)=\Theta^t_*(g)\hspace{1cm}\mbox{for some }g\in\Lambda(FG)^\times.
\]
\end{prop}

Write $L:=F^h$ and $\zeta=\zeta_{p,F}$ for the chosen primitive $p$-root of unity in $F^c$.

\begin{lem}\label{lem:12.3} There exists $x\in F(\zeta)$ such that $v_{L(\zeta)}(x^{1/p}-1)=1$ and
\[
L(\zeta)=F(\zeta,x^{1/p}).
\]
\end{lem}
\begin{proof}See \cite[Section 3 and the discussion following Lemma 8]{P}. The assumptions that $F/\mathbb{Q}_p$ is unramified with $p$ odd and $L\subset F_{p,2}$ are required.
\end{proof}

We summarize the set-up in the diagram below, where the numbers indicate the degrees of the extensions.
\[
\begin{tikzpicture}[baseline=(current bounding box.center)]
\node at (0,0) [name=F] {$F$};
\node at (5,1) [name=F'] {$F(\zeta)$};
\node at (0,2.5) [name=L] {$L$};
\node at (5,3.5) [name=L'] {$L(\zeta)$};
\node at (6.9,3.5) {$=F(\zeta,x^{1/p})$};
\node at (11.1,1.5) {$v_{L(\zeta)}(x^{1/p}-1)=1$};
\node at (10,2.5) {$L:=F^h$};
\path[font=\small]
(F') edge node[auto] {$p-1$} (F)
(F) edge node[auto] {$p$} (L)
(L') edge node[above] {$p-1$} (L)
(L') edge node[auto] {$p$} (F');
\end{tikzpicture}
\]

\begin{definition}\label{def:12.4}Let $\mathbb{F}_p:=\mathbb{Z}/p\mathbb{Z}$. For each $i\in\mathbb{F}_p$, let $c(i)\in\{\frac{1-p}{2},\dots,\frac{p-1}{2}\}$ be the unique element which represents $i$, and if $i\in\mathbb{F}_p^\times$ we write $i^{-1}$ for the multiplicative inverse of $i$ in $\mathbb{F}_p^\times$. For each $i\in\mathbb{F}_p^\times$, define
\[
\omega_i\in\mbox{Gal}(L(\zeta)/L);\hspace{1em}\omega_i(\zeta):=\zeta^{c(i^{-1})}.
\]
Moreover, define $x_{i}:=\omega_{i}(x)$ and
\[
x_{i}^{1/p}:=\omega_{i}(x^{1/p}),
\]
which is clearly a $p$-th root of $x_i$. We will also write $y_i=x_i^{1/p}$.
\end{definition}

Consider the element
\[
\alpha:=\frac{1}{p}\left(\sum_{k\in\mathbb{F}_p}\prod_{i\in\mathbb{F}_p^\times}y_i^{c(ik)}\right)\\
=\frac{1}{p}\left(1+\prod_{i\in\mathbb{F}_p^\times}y_i^{c(i)}+\cdots+\prod_{i\in\mathbb{F}_p^\times}y_{i}^{c(i(p-1))}\right).
\]
We will show that $a\in\mbox{Map}(G,F^c)$ defined by
\begin{equation}\label{eq:12.1}
a(s):=\begin{cases}
\omega(\alpha) & \mbox{if }s=h(\omega)\mbox{ for }\omega\in\Omega_F\\
0 & \mbox{otherwise}
\end{cases}
\end{equation}
satisfies the conclusion of Proposition~\ref{prop:12.2}. The definition of $\alpha$ is motivated by the definition of $g$ in (\ref{eq:12.3}), the computation (\ref{eq:12.4}), and the formula (\ref{eq:6.7}).

\begin{lem}\label{lem:12.5}We have $\alpha\in L$.
\end{lem}
\begin{proof}By definition, we have $y_i\in L(\zeta)$ for all $i\in\mathbb{F}_p^\times$. So, clearly $\alpha\in L(\zeta)$ and we have $\alpha\in L$ if and only if $\alpha$ is fixed by the action of $\mbox{Gal}(L(\zeta)/L)$.

Now, a non-trivial element in $\mbox{Gal}(L(\zeta)/L)$ is equal to $\omega_{j}$ for some $j\in\mathbb{F}_p^\times$. Moreover, notice that $c(i^{-1})c(j^{-1})=c((ji)^{-1})$ and hence $\omega_j(y_i)=y_{ji}$ for all $i\in\mathbb{F}_p^\times$. Hence, for each $k\in\mathbb{F}_p$ we have
\[
\omega_{j}\left(\prod_{i\in\mathbb{F}_p^\times}y_{i}^{c(ik)}\right)
=\prod_{i\in\mathbb{F}_p^\times}y_{ji}^{c(ik)}\\
=\prod_{i\in\mathbb{F}_p^\times}y_{i}^{c(i\cdot j^{-1}k)}.
\]
This implies that $\omega_j$ permutes the summands
\[
1,\prod_{i\in\mathbb{F}_p^\times}y_{i}^{c(i)},\dots,\prod_{i\in\mathbb{F}_p^\times}y_{i}^{c(i(p-1))}
\]
in the definition of $\alpha$. This shows that $\omega_j(\alpha)=\alpha$ and so $\alpha\in L$.
\end{proof}

Next, we use a valuation argument to show that $\alpha\in A_{L/F}$.

\begin{lem}\label{lem:12.6}For all $i\in\mathbb{F}_p^\times$, we have $v_{L(\zeta)}(y_i)=0$ and 
\[
v_{L(\zeta)}(y_i^n-1)\geq 1\hspace{1cm}\mbox{for all }n\in\mathbb{Z}.
\]
\end{lem}
\begin{proof}Recall that $x^{1/p}-1$ is a uniformizer in $L(\zeta)$. For each $i\in\mathbb{F}_p^\times$, we have
\[
y_i-1=x_i^{1/p}-1=\omega_{i}(x^{1/p}-1),
\]
which show that $v_{L(\zeta)}(y_i-1)=1$ and in particular $v_{L(\zeta)}(y_i)=0$.

Note that the second claim is obvious for $n=0$. For $n\in\mathbb{Z}^+$, we have
\[
v_{L(\zeta)}(y_{i}^{n}-1)=v_{L(\zeta)}(y_{i}-1)+v_{L(\zeta)}(y_{i}^{n-1}+\cdots+y_{i}+1)\geq 1+0.
\]
For $n\in\mathbb{Z}^-$, use the above to deduce that
\[
v_{L(\zeta)}(y_{i}^{n}-1)=v_{L(\zeta)}(y_i^{n})+v_{L(\zeta)}(1-y_i^{-n})\geq 0+1.
\]
This proves the lemma.
\end{proof}

\begin{prop}\label{prop:12.7}We have $\alpha\in A_{L/F}$.
\end{prop}
\begin{proof}Observe that $L/F$ has ramification index $p$ because it has degree $p$ by hypothesis and $F_{p,2}/F$ is totally ramified. Then, by Proposition~\ref{prop:12.1} we have
\[
v_L(A_{L/F})=1-p.
\]
On the other hand, observe that $v_L(p)=p$ because $F/\mathbb{Q}_p$ is unramified, so
\begin{align*}
v_L(\alpha)&=v_L\left(\sum_{k\in\mathbb{F}_p}\prod_{i\in\mathbb{F}_p^\times}y_i^{c(ik)}\right)
-p\\
&=v_L\left(\sum_{k\in\mathbb{F}_p}\left(\prod_{i\in\mathbb{F}_p^\times}y_i^{c(ik)}-1\right)+p\right)-p\\
&\geq \min\Bigg\{v_L\left(\sum_{k\in\mathbb{F}_p}\left(\prod_{i\in\mathbb{F}_p^\times}y_i^{c(ik)}-1\right)\right), p\Bigg\}-p.
\end{align*}
But identifying $\mathbb{F}_p^\times$ with $\{1,2\dots,p-1\}$, we see that for each $k\in\mathbb{F}_p$ we have
\[
\prod_{i\in\mathbb{F}_p^\times}y_i^{c(ik)}-1
=\sum_{i=1}^{p-2}\left(\left(\prod_{l=i+1}^{p-1}y_l^{c(lk)}\right)(y_i^{c(ik)}-1)\right)+(y_{p-1}^{c((p-1)k)}-1).
\]
It then follows from Lemma~\ref{12.6} that
\[
v_{L(\zeta)}\left(\prod_{i\in\mathbb{F}_p^\times}y_i^{c(ik)}-1\right)\geq 1,
\]
and in particular, we have
\[
v_L\left(\sum_{k\in\mathbb{F}_p}\left(\prod_{i\in\mathbb{F}_p^\times}y_i^{c(ik)}-1\right)\right)\geq 1.
\]
This shows that $v_L(\alpha)\geq 1-p$, whence $\alpha\in A_{L/F}$, as claimed.
\end{proof}

Next, we compute the Galois conjugates of $\alpha$ in $L/F$. Observe that there is a canonical isomorphism
\[
\mbox{Gal}(L(\zeta)/F)\simeq \mbox{Gal}(L/F)\times\mbox{Gal}(F(\zeta)/F),
\]
which restricts to an isomorphism
\begin{equation}\label{eq:12.2}
\mbox{Gal}(L(\zeta)/F(\zeta))\simeq\mbox{Gal}(L/F).
\end{equation}
Let $\tau\in\mbox{Gal}(L/F)$ be the generator which is identified with
\[
\tilde{\tau}\in\mbox{Gal}(L(\zeta)/F(\zeta));\hspace{1em}\tilde{\tau}(x^{1/p}):=\zeta^{-1}x^{1/p}
\]
via (\ref{eq:12.2}). So, the elements $\tau$ and $\tilde{\tau}$ are equal when restricted to $L$.

\begin{prop}\label{prop:12.8}For all $j,k\in\mathbb{F}_p$, we have
\[
\tilde{\tau}^{c(j)}\left(\prod_{i\in\mathbb{F}_p^\times}y_i^{c(ik)}\right)=\zeta^{c(jk)}\cdot\prod_{i\in\mathbb{F}_p^\times}y_i^{c(ik)}.
\]
In particular, this implies that for all $j\in\mathbb{F}_p$ we have
\[
\tau^{c(j)}(a)=\frac{1}{p}\sum_{k\in\mathbb{F}_p}\left(\zeta^{c(jk)}\prod_{i\in\mathbb{F}_p^\times}y_i^{c(ik)}\right).
\]
\end{prop}
\begin{proof}Let $j,k\in\mathbb{F}_p$. Since $\mbox{Gal}(L(\zeta)/F)$ is abelian, for any $i\in\mathbb{F}_p^\times$ we have
\begin{align*}
\tilde{\tau}^{c(j)}(y_i)
&=(\tilde{\tau}^{c(j)}\circ\omega_i)(x^{1/p})\\
&=(\omega_i\circ\tilde{\tau}^{c(j)})(x^{1/p})\\
&=\omega_i(\zeta^{c(-1)c(j)}x^{1/p})\\
&=\zeta^{c(i^{-1})c(-1)c(j)}y_i.
\end{align*}
Since $c(i^{-1})c(-1)c(j)c(ik)\equiv -c(jk)$ (mod $p$), we deduce that
\begin{align*}
\tilde{\tau}^{c(j)}\left(\prod_{i\in\mathbb{F}_p^\times}y_i^{c(ik)}\right)
&=\prod_{i\in\mathbb{F}_p^\times}\zeta^{c(i^{-1})c(-1)c(j)c(ik)}y_i^{c(ik)}\\
&=\zeta^{-c(jk)(p-1)}\prod_{i\in\mathbb{F}_p^\times}y_i^{c(ik)}\\
&=\zeta^{c(jk)}\prod_{i\in\mathbb{F}_p^\times}y_i^{c(ik)}.
\end{align*}
This proves the first claim, and the second follows directly from the first.
\end{proof}

\begin{proof}[Proof of Propostion~\ref{prop:12.2}] Let $a\in\mbox{Map}(G,F^c)$ be as in (\ref{eq:12.1}). Since $L=F^h$, clearly $a\in F_h$. In fact, we have $a\in A_h$ because $\alpha\in A_{L/F}$ by Proposition~\ref{prop:12.7}.

Recall that $h$ induces an isomorphism $\mbox{Gal}(L/F)\simeq h(\Omega_F)$. Let $\omega_\tau\in\Omega_F$ be a lift of $\tau$ and $t:=h(\omega_\tau)$. Then $h(\Omega_F)=\langle t\rangle$ and has order $p$. Define
\begin{equation}\label{eq:12.3}
g\in\Lambda(FG)^\times;\hspace{1em}
g(s):=\begin{cases}
x_i & \mbox{if $s=t^{c(i)}$ for $i\in\mathbb{F}_p^\times$}\\
1 & \mbox{otherwise}.
\end{cases}
\end{equation}

First we verify that $g$ indeed preserves $\Omega_F$-action. It will be helpful to recall Definition~\ref{def:7.3}. So let $\omega\in\Omega_F$ and $s\in G(-1)$. If $s\neq t^{c(i)}$ for all $i\in\mathbb{F}_p^\times$, then the same is true for $\omega\cdot s$ because $\langle\omega\cdot s\rangle=\langle s\rangle$. In this case, we have
\[
g(\omega\cdot s)=1=\omega(1)=\omega(g(s)).
\]
If $s=t^{c(i)}$ for some $i\in\mathbb{F}_p^\times$, then $s$ has order $p$. Since $x_i\in F(\zeta)$, it is enough to consider $\omega|_{F(\zeta)}$. If $\omega|_{F(\zeta)}=\mbox{id}_{F(\zeta)}$, then $\omega$ fixes both $x_i$ and the elements of order dividing $p$ in $G$, so
\[
g(\omega\cdot s)=g(s)=\omega(g(s)).
\]
If $\omega|_{F(\zeta)}\neq\mbox{id}_{F(\zeta)}$, then $\omega|_{F(\zeta)}=\omega_j|_{F(\zeta)}$ for some $j\in\mathbb{F}_p^\times$. Since $\omega^{-1}_j(\zeta)=\zeta^{c(j)}$ by definition and $|t|=p$, we see that
\[
\omega\cdot t^{c(i)}
=t^{c(i)\kappa(\omega^{-1})}
=t^{c(i)c(j)}
=t^{c(ji)}.
\]
Using the identify $\omega_j\omega_i=\omega_{ji}$, we deduce that
\begin{align*}
g(\omega\cdot t^{c(i)})
&=g(t^{c(ji)})\\
&=x_{ji}\\
&=\omega_j(x_i)\\
&=\omega(g(t^{c(i)})).
\end{align*}
This proves that $g$ preserves $\Omega_F$-action.
 
Next we check that $\Theta^t_*(g)=r_G(a)$. Given $\chi\in\hat{G}_F$, let $k\in\mathbb{F}_p$ be such that $\chi(t)=\zeta^{c(k)}$. Then, we have $\langle\chi,t^{c(i)}\rangle_*=c(ik)/p$ by Definitions~\ref{def:7.1} and~\ref{def:12.4}. On one hand, we have
\begin{eqnarray}\label{eq:12.4}
\Theta^t_*(g)(\chi)&=&\prod_{s\in G}g(s)^{\langle\chi,s\rangle_*}\\\notag
&=&\prod_{i\in\mathbb{F}_p^\times}x_i^{\langle\chi,t^{c(i)}\rangle_*}\\\notag
&=&\prod_{i\in\mathbb{F}_p^\times}y_i^{c(ik)}.\notag
\end{eqnarray}
On the other hand, by (\ref{eq:12.1}) and Proposition~\ref{prop:12.8}, we have that
\begin{align*}
\textbf{r}_G(a)(\chi)&=\sum_{s\in G}a(s)\chi(s)^{-1}\\
&=\sum_{j\in\mathbb{F}_p}\tau^{c(j)}(\alpha)\chi(t^{c(j)})^{-1}\\
&=\frac{1}{p}\sum_{j\in\mathbb{F}_p}\sum_{l\in\mathbb{F}_p}\left(\zeta^{c(jl)}\prod_{i\in\mathbb{F}_p^\times} y_i^{c(il)}\right)\zeta^{-c(jk)}\\
&=\frac{1}{p}\sum_{l\in\mathbb{F}_p}\left(\prod_{i\in\mathbb{F}_p^\times} y_i^{c(il)}\sum_{j\in\mathbb{F}_p}\zeta^{c(jl)-c(jk)}\right).
\end{align*}
But $c(jl)-c(jk)\equiv c(j(l-k))$ (mod $p$) and
\[
\sum_{j\in\mathbb{F}_p}\zeta^{c(j(l-k))}
=\begin{cases}
p & \mbox{if }l-k=0\\
0 & \mbox{otherwise}.
\end{cases}
\]
It follows that
\[
\textbf{r}_G(a)(\chi)=\prod_{i\in\mathbb{F}_p^\times}y_i^{c(ik)}
\]
also. We conclude from the identification (\ref{eq:6.8}) that $r_G(a)=\Theta^t_*(g)$.

Finally, since $r_G(a)=\Theta^t_*(g)$, it follows from Proposition~\ref{prop:7.5} that
\[
\textbf{r}_G(a)\textbf{r}_G(a)^{[-1]}=1.
\]
Since $a\in A_h$, we deduce from Corollary~\ref{cor:2.9} that $A_h=\mathcal{O}_FG\cdot a$. This proves that $a$ satisfies all of the properties claimed in Proposition~\ref{prop:12.2}.
\end{proof}


\section{Decomposition of Local Wild Resolvends I}\label{s:13}

Let $F$ be a finite extension of $\mathbb{Q}_p$. Theorem~\ref{thm:13.1} below is analogous to Theorem~\ref{thm:8.5}, except now we consider a wildly and weakly ramified homomorphism instead of a tame one.

\begin{thm}\label{thm:13.1}Let $h\in\mbox{Hom}(\Omega_F,G)$ be weakly ramified such that $F^h/F$ has ramification index $p$. Assume in addition that $F/\mathbb{Q}_p$ is unramified and that $p$ is odd. If $A_h=\mathcal{O}_FG\cdot a$, then
\[
r_G(a)=u\Theta^t_*(g)
\]
for some $u\in\mathcal{H}(\mathcal{O}_FG)$ and $g\in\Lambda(FG)^\times$.
\end{thm}
\begin{proof}Since $F/\mathbb{Q}_p$ is unramified, we may choose $p$ to be a uniformizer of $F$. From Proposition~\ref{prop:11.3}, there exists a factorization $h=h^{nr}h^{tot}$ of $h$ such that $F^{h^{tot}}\subset F_{p,2}$. Moreover, by Lemma~\ref{lem:4.3}, the ramification index of $F^{h^{tot}}/F$ is equal to that of $F^h/F$, which is $p$ by hypothesis. 

On one hand, since $p$ is odd, the remarks above and Proposition~\ref{prop:12.2} imply that there is an element $a_{tot}\in F_{h^{tot}}$ such that $A_{h^{tot}}=\mathcal{O}_FG\cdot a_{tot}$ and
\[
r_G(a_{tot})=\Theta^t_*(g)\hspace{1cm}\mbox{for some }g\in\Lambda(FG)^\times.
\]
On the other hand, by Corollary~\ref{cor:2.10} and (\ref{eq:6.3}), there exists  $a_{nr}\in F_{h^{nr}}$ such that $\mathcal{O}_{h^{nr}}=\mathcal{O}_FG\cdot a_{nr}$ and
\[
r_G(a_{nr})=u'\hspace{1cm}\mbox{for some }u'\in\mathcal{H}(\mathcal{O}_FG)
\]
Proposition~\ref{prop:4.4} then yields an element $a'\in F_h$ such that $A_h=\mathcal{O}_FG\cdot a'$ and
\[
\textbf{r}_{G}(a')=\textbf{r}_{G}(a_{nr})\textbf{r}_{G}(a_{tot}).
\]
Since $A_h=\mathcal{O}_FG\cdot a$ also, we have $a=\gamma\cdot a'$ for some $\gamma\in(\mathcal{O}_FG)^{\times}$. So
\[
r_{G}(a)=rag(\gamma)r_{G}(a')=(rag(\gamma) u')\Theta_{*}^{t}(g),
\]
where $u:=rag(\gamma) u'\in\mathcal{H}(\mathcal{O}_FG)$ and this proves the claim.
\end{proof}


\section{Proof of Theorem~\ref{thm:1.4}}\label{s:14}

\begin{proof}[Proof of Theorems~\ref{thm:1.4}]Let $h\in\mbox{Hom}(\Omega_K,G)$ be weakly ramified and $V$ the set of primes in $\mathcal{O}_K$ for which are $h_v$ is wildly ramified. Assume in addition that conditions (i) and (ii) hold for all $v\in V$. 

Let $K_h=KG\cdot b$ and $A_{h_v}=\mathcal{O}_{K_v}G\cdot a_v$ for $v\in M_K$ as in Remark~\ref{rem:3.5}. Then, we have $j(c)=\mbox{cl}(A_h)$, where $c\in J(KG)$ satisfies $a_v=c_v\cdot b$ for each $v\in M_K$. In particular, we jave
\[
rag(c_v)=r_G(b)^{-1}r_G(a_v)\hspace{1cm}\mbox{for }v\in M_K.
\]
Then, by (\ref{eq:10.1}) and (\ref{eq:6.1'}), it suffices to show that for all $v\in M_K$ we have
\begin{equation}\label{eq:14.1}
r_G(a_v)\in\mathcal{H}(\mathcal{O}_{K_v}G)\Theta^t_*(\Lambda(K_vG)^\times).
\end{equation}
For $v\notin V$, the above holds by Theorem~\ref{thm:8.5}. For $v\in V$, let $p$ be the rational prime below it. Notice that $p$ must be odd since $G$ has odd order. Moreover, condition (ii) implies that the ramification index of $(K_v)^{h_v}/K_v$ is equal to $p$. Since $K_v/\mathbb{Q}_p$ is unramified by condition (i), we see from Theorem~\ref{thm:13.1} that (\ref{eq:14.1}) holds. Hence, indeed $\mbox{cl}(A_h)\in\mathcal{A}^t(\mathcal{O}_KG)$.
\end{proof}


\section{Valuation of Local Wild Resolvents}\label{s:15}

Let $F$ be a finite extension of $\mathbb{Q}_p$. In this section, we write $\zeta=\zeta_{p,F}$ for the chosen primitive $p$-th root of unity in $F^c$. Moreover, we assume that $\zeta\notin F$.

\begin{definition}\label{def:15.1}Let $a\in\mbox{Map}(G,F^c)$ and $\chi\in\hat{G}_F$. Define
\[
\left(a\mid\chi\right):=\textbf{r}_G(a)(\chi),
\]
called the \emph{resolvent of $a$ at $\chi$}.
\end{definition}

In the remaining of this section, we prove Proposition~\ref{prop:15.2} below.

\begin{prop}\label{prop:15.2}Let $h\in\mbox{Hom}(\Omega_F,G)$ be wildly and weakly ramified. Under the assumption that $\zeta\notin F$, if $A_h=\mathcal{O}_FG\cdot a$, then
\[
\left(a\mid\chi\right)\in\mathcal{O}_{F^c}^\times
\hspace{1cm}\mbox{for all }\chi\in\hat{G}_F.
\]
\end{prop}

\begin{lem}\label{lem:15.3}Let $L/F$ be a finite Galois extension and denote by $G(L/F)_n$ the $n$-th ramification group of $L/F$. 
\begin{enumerate}[(a)]
\item The quotients $G(L/F)_n/G(L/F)_{n+1}$ are elementary $p$-abelian for $n\geq 1$.
\item If $L/F$ is a wildly and weakly ramified abelian extension, then its ramification index is equal to $p^r$ for some $r\in\mathbb{Z}^+$ and $G(L/F)_0\simeq(\mathbb{Z}/p\mathbb{Z})^r$.
\end{enumerate}
\end{lem}
\begin{proof}For (a), see \cite[Chapter IV Proposition 7 Corollary 3]{S}, for example. For (b), observe that if $L/F$ is weakly ramified, then (a) implies that
\[
G(L/F)_1\simeq G(L/F)_1/G(L/F)_2\simeq (\mathbb{Z}/p\mathbb{Z})^r\hspace{1cm}\mbox{for some }r\in\mathbb{Z}^+.
\]
If $L/F$ is abelian and wildly ramified in addition, then Lemma~\ref{lem:4.2} (c) implies that $G(L/F)_0=G(L/F)_1$, from which the claims follow.
\end{proof}

\begin{proof}[Proof of Proposition~\ref{prop:15.2}]
Since $F/\mathbb{Q}_p$ is unramified, we may choose $p$ to be a uniformizer of $F$. By Proposition~\ref{prop:11.3}, there exists a factorization $h=h^{nr}h^{tot}$ of $h$ such that $F^{h^{tot}}\subset F_{p,2}$. Notice that $h^{tot}=(h^{nr})^{-1}h$. Since $h$ is wildly and weakly ramified, it follows from Lemma~\ref{lem:4.3} that the same is true for $h^{tot}$.

Consider $F^{h^{tot}}/F$, which is totally, wildly, and weakly ramified.  Lemma~\ref{lem:15.3} (b) then implies that $h^{tot}(\Omega_F)$, which is isomorphic to $\mbox{Gal}(F^{h^{tot}}/F)$, has exponent $p$. Let $\alpha\in F^h$ be such that $A_{F^h/F}=\mathcal{O}_F\mbox{Gal}(F^h/F)\cdot\alpha$. It is not hard to see that $a_{tot}\in\mbox{Map}(G,F^c)$ given by
 \[
a_{tot}(s):=
\begin{cases}
\omega(\alpha) &\mbox{if $s=h^{tot}(\omega)$ for $\omega\in\Omega_F$}\\
0 & \mbox{otherwise}
\end{cases}
\]
is an element in $F_{h^{tot}}$ and that $A_{h^{tot}}=\mathcal{O}_FG\cdot a_{tot}$.

Now, by Corollary~\ref{cor:2.10}, there is $a_{nr}\in F_{h^{nr}}$ such that $\mathcal{O}_{h^{nr}}=\mathcal{O}_{F}G\cdot a_{nr}$ and
\[
\textbf{r}_G(a_{nr})=u\hspace{1cm}\mbox{for some }u\in(\mathcal{O}_{F^c}G)^\times.
\]
From Proposition~\ref{prop:4.4}, we obtain $a'\in F_h$ such that $A_h=\mathcal{O}_FG\cdot a'$ and
\[
\textbf{r}_{G}(a')=\textbf{r}_{G}(a_{nr})\textbf{r}_{G}(a_{tot}).
\]
Since $A_h=\mathcal{O}_FG\cdot a$ also, we have $a=\gamma\cdot a'$ for some $\gamma\in(\mathcal{O}_FG)^\times$. So
\[
(a\mid\chi)=\gamma(\chi)(a_{nr}\mid\chi)(a_{tot}\mid\chi)\hspace{1cm}\mbox{for all }\chi\in\hat{G}_F.
\]
Given any $\chi\in\hat{G}_F$, it is clear that $\gamma(\chi),(a_{nr}\mid\chi)\in\mathcal{O}_{F^c}^\times$, Hence, it remains to show that $(a_{tot}\mid\chi)\in\mathcal{O}_{F^c}^\times$ for all $\chi\in\hat{G}_F$ also.

To that end, recall that $A_{h^{tot}}=\mathcal{O}_FG\cdot a_{tot}$. Hence, we have
\[
\textbf{r}_G(a_{tot})\textbf{r}_G(a_{tot})^{[-1]}\in(\mathcal{O}_FG)^\times
\]
from Corollary~\ref{cor:2.9}. In particular, this implies that 
\begin{equation}\label{eq:15.1}
(a_{tot}\mid\chi)(a_{tot}\mid\chi^{-1})\in\mathcal{O}_{F^c}^\times\hspace{1cm}\mbox{for all }\chi\in\hat{G}.
\end{equation}
Moreover, we have $[F(\zeta):F]=p-1$ because $\zeta\notin F$. Since $[F^{h^{tot}}:F]$ equals a power of $p$, we have a natural isomorphism
\[
\mbox{Gal}(F^{h^{tot}}(\zeta)/F)\simeq\mbox{Gal}(F(\zeta)/F)\times\mbox{Gal}(F^{h^{tot}}/F).
\]
Let $\omega\in\Omega_F$ be a lift of the element
\[
(\zeta\mapsto\zeta^{-1})\times\mbox{id}_{F^{h^{tot}}}.
\]
Notice that $a_{tot}(s)\in F^{h^{tot}}$ and so is fixed by $\omega$ for all $s\in G$. Moreover, since $a_{tot}(s)=0$ for all $s\notin h^{tot}(\Omega_F)$ and elements in $h^{tot}(\Omega_F)$ have order $p$, we have
\[
(a_{tot}\mid\chi^{-1})=\omega(a_{tot}\mid\chi)
\hspace{1cm}\mbox{for all }\chi\in\hat{G}_F.
\]
Hence, given any $\chi\in\hat{G}$ we have
\[
v_{F^{h^{tot}}(\zeta)}((a_{tot}\mid\chi^{-1}))=v_{F^{h^{tot}}(\zeta)}((a_{tot}\mid\chi)).
\]
This together with (\ref{eq:15.1}) shows that $(a_{tot}\mid\chi)\in\mathcal{O}_{F^c}^\times$, as desired.
\end{proof}


\section{Decomposition of Local Wild Resolvends II}\label{s:16}

Let $F$ be a finite extension of $\mathbb{Q}_p$ and $M(FG)$ the maximal $\mathcal{O}_F$-order contained in $FG$. Theorem~\ref{thm:16.1} below is analogous to Theorem~\ref{thm:13.1}, except now there is no restriction on the ramification index.

\begin{thm}\label{thm:16.1}Let $h\in\mbox{Hom}(\Omega_F,G)$ be wildly and weakly ramified. Assume in addition that $F/\mathbb{Q}_p$ is unramified and that $p$ is odd. If $A_h=\mathcal{O}_FG\cdot a$, then
\[
r_G(a)= rag(\gamma)u\Theta^t_*(g)
\]
for some $\gamma\in M(FG)^\times$, $u\in\mathcal{H}(\mathcal{O}_FG)$, and $g\in\Lambda(FG)^\times$.
\end{thm}
\begin{proof}
As in the proof of Proposition~\ref{prop:15.2}, there is a factorization $h=h^{nr}h^{tot}$ of $h$ such that $F^{h^{tot}}\subset F_{p,2}$, and we know that $h^{tot}(\Omega_F)$ has exponent $p$. So
\[
h^{tot}(\Omega_F)\simeq H_1\times H_2\times\cdots H_r
\]
for subgroups $H_1,H_2,\dots,H_r$ each of order $p$. For each $i=1,2,\dots,r$, define
\[
h_i\in\mbox{Hom}(\Omega_F,G);\hspace{1em}
h_i(\omega):=\uppi_i(h^{tot}(\omega)),
\]
where $\uppi_i:h^{tot}(\Omega_F)\longrightarrow H_i$ denotes the natural projection map. Clearly
\[
h^{tot}=h_1h_2\cdots h_r.
\]

Moreover, observe that for each $i=1,2,\dots,r$, we have
\[
F^{h_i}\subset F^{h^{tot}}\subset F_{p,2}.
\]
The extension $F^{h_i}/F$ has ramification index $p$ since it is totally ramified and
\[
[F^{h_i}:F]=|h_i(\Omega_F)|=p.
\]
Since $F/\mathbb{Q}_p$ is unramified and $p$ is odd, by Proposition~\ref{prop:12.2} there is $a_i\in F_{h_i}$ such that $A_{h_i}=\mathcal{O}_FG\cdot a_i$ and
\[
r_G(a_i)=\Theta^t_*(g_i)\hspace{1cm}\mbox{for some }g_i\in\Lambda(FG)^\times.
\]
On the other hand, by Corollary~\ref{cor:2.10} and (\ref{eq:6.1''}), there exists $a_{nr}\in F_{h^{nr}}$ such that $\mathcal{O}_{h^{nr}}=\mathcal{O}_FG\cdot a_{nr}$ and
\[
r_G(a_{nr})=u\hspace{1cm}\mbox{for some }u\in\mathcal{H}(\mathcal{O}_FG).
\]

Now, let $a'\in\mbox{Map}(G,F^c)$ be such that
\[
\textbf{r}_G(a')=\textbf{r}_G(a_{nr})\textbf{r}_G(a_1)\cdots\textbf{r}_G(a_r).
\]
From (\ref{eq:2.1}) and Proposition~\ref{prop:2.6}, we know that $a'\in F_h$ and that $F_h=FG\cdot a'$. Since $F_h=FG\cdot a$ also, we have $a=\gamma\cdot a'$ for some $\gamma\in (FG)^\times$. So
\[
r_G(a)=rag(\gamma)r_G(a')=rag(\gamma)u\Theta^t_*(g),
\]
where $g:=g_1g_2\cdots g_r\in\Lambda(FG)^\times$. It remains to show that $\gamma\in M(FG)^\times$.

To that end, observe that
\[
M(FG)^\times=\mbox{Map}(\hat{G}_F,\mathcal{O}_F^\times)
\]
via identification (\ref{eq:6.3}). Moreover, given any $\chi\in\hat{G}_F$, we have
\[
(a\mid\chi)=\gamma(\chi)(a_{nr}\mid\chi)(a_1\mid\chi)\cdots (a_r\mid\chi),
\]
where clearly $(a_{nr}\mid\chi)\in\mathcal{O}_{F^c}^\times$. Since $F$ is unramified over $\mathbb{Q}_p$ and so does not contain any primitive $p$-roots of unity, Proposition~\ref{prop:15.2} implies that
\[
(a\mid\chi),(a_1\mid\chi),\dots,(a_r\mid\chi)\in\mathcal{O}_{F^c}^\times
\]
also. This shows that $\gamma(\chi)\in\mathcal{O}_{F}^\times$ and so indeed $\gamma\in M(FG)^\times$, as desired.
\end{proof}

\begin{remark}\label{rem:16.2}The element $a\in F_h$ will not be a generator of $A_h$ over $\mathcal{O}_FG$ in general because $h_1,h_2,\dots,h_r$ are all ramified. For example, Proposition~\ref{prop:4.4} does not apply here. This is why the assumption that $F^h/F$ has ramification index $p$ is needed in Theorem~\ref{thm:13.1}.
\end{remark}


\section{Proof of Theorem~\ref{thm:1.5}}\label{s:17}

\begin{proof}[Proof of Theorems~\ref{thm:1.5}]Let $h\in\mbox{Hom}(\Omega_K,G)$ be weakly ramified and $V$ the set of primes in $\mathcal{O}_K$ for which are $h_v$ is wildly ramified. Assume in addition that $v$ is unramified over $\mathbb{Q}$ for all $v\in V$.

Let $K_h=KG\cdot b$ and $A_{h_v}=\mathcal{O}_{K_v}G\cdot a_v$ for $v\in M_K$ as in Remark~\ref{rem:3.5}. Then, we have $j(c)=\mbox{cl}(A_h)$, where $c\in J(KG)$ satisfies $a_v=c_v\cdot b$ for each $v\in M_K$. Moreover, by Theorem~\ref{thm:3.3} we have
\[
\ker(\Psi)=\frac{\partial(KG)^\times U(M(KG))}{\partial(KG)^\times U(\mathcal{O}_FG)}.
\]
Hence, to show that $\Psi(\mbox{cl}(A_h))\in\Psi(\mathcal{A}^t(\mathcal{O}_KG))$, it suffices to show that there exists $\gamma\in U(M(KG))$ such that $j(c)j(\gamma)\in \mathcal{A}^t(\mathcal{O}_KG)$. Since
\[
rag(c_v)=r_G(b)^{-1}r_G(a_v)\hspace{1cm}\mbox{for }v\in M_K,
\]
by (\ref{eq:10.1}) and (\ref{eq:6.1'}), it is enough to show that for all $v\in M_K$, we have
\begin{equation}\label{eq:17.1}
r_G(a_v)\in rag(M(K_vG)^\times)\mathcal{H}(\mathcal{O}_{K_v}G)\Theta^t_*(\Lambda(K_vG)^\times),
\end{equation}
where $M(K_vG)$ is the maximal $\mathcal{O}_{K_v}$-order in $K_vG$. For $v\notin V$, the above holds by Theorem~\ref{thm:8.5}. For $v\in V$, let $p$ be the rational prime below it, which must be odd since $G$ has odd order. Moreover, since $K_v/\mathbb{Q}_p$ is unramified, Theorem~\ref{thm:16.1} implies that (\ref{eq:17.1}) holds. So, indeed $\Psi(\mbox{cl}(A_h))\in\Psi(\mathcal{A}^t(\mathcal{O}_KG))$.
\end{proof}


\section{Acknowledgements}\label{s:18}

I would like to thank my advisor Adebisi Agboola for bringing this problem to my attention and giving me many helpful ideas. I am very grateful for his patience and all of his suggestions in the process of writing this paper.

\end{document}